\documentclass[article,12pt]{amsart}
\usepackage{amscd}
\usepackage{bbm}
\usepackage{dsfont}
\usepackage{enumerate}
\usepackage{latexsym}
\usepackage{srcltx}
\usepackage{amsfonts}
\usepackage{amssymb}
\usepackage{datetime}
\usepackage{setspace}
\usepackage{enumerate}
\usepackage{amsthm}
\usepackage{graphicx}
\usepackage{epstopdf}

\newtheorem{theorem}{Theorem}[section]
\newtheorem{proposition}[theorem]{Proposition}
\newtheorem{lemma}[theorem]{Lemma}
\newtheorem{corollary}[theorem]{Corollary}

\theoremstyle{definition}

\newtheorem{definition}[theorem]{Definition}

\newtheorem{remark} [theorem] {Remark}

\newtheorem{problem} [theorem] {Problem}
\input epsf
\begin{document}

\title{ Spectrum of Weighted Composition Operators \\
Part III \\
Essential spectra of some disjointness preserving operators
on Banach lattices.}

\author{A. K. Kitover}

\address{Community College of Philadelphia, 1700 Spring Garden St., Philadelphia, PA, USA}

\email{akitover@ccp.edu}


\date{\today}


\maketitle

\markboth{A.K.Kitover}{Spectrum of weighted composition operators. III}

\bigskip

\begin{abstract} We describe essential (in particular Fredholm and semi-Fredholm) spectra of operators on Banach lattices of the form
$T=wU$, where $w$ is a central operator and $U$ is a disjointness preserving operator such that its spectrum $\sigma(U)$ is a subset of the unit circle.
\end{abstract}

\section{Introduction}

We recall the following well known definition (see e.g.~\cite[Ch. IV, Sec. 5, page 230]{Ka}).

\begin{definition} \label{d1} A bounded linear operator $T$ on a Banach space $X$ is called \textit{semi-Fredholm} if $TX$ is closed in $X$ and either
$nul(T) = \dim \ker (T) < \infty$ or ${def}(T) = \dim (X/TX) < \infty$.

The operator $T$ is called \textit{Fredholm} if $nul(T) + def(T) < \infty$

\end{definition}

In correspondence with Definition~\ref{d1} the semi-Fredholm and Fredholm spectra of a bounded linear operator $T$ are defined as follows. ($\sigma(T)$ as usual means the spectrum of a bounded linear operator $T$ on a Banach space $X$)

\begin{definition} \label{d2} Let $T$ be a bounded linear operator on a Banach space $X$. The semi-Fredholm spectrum of $T$ is
$$\sigma_{sf}(T) = \{\lambda \in \sigma(T) : \; \mathrm{the} \; \mathrm{operator} \; \lambda I - T \, \mathrm{is \, not\; semi-Fredholm} \}.$$
The Fredholm spectrum of $T$ is
$$\sigma_{f}(T) = \{\lambda \in \sigma(T) : \; \mathrm{the} \; \mathrm{operator} \; \lambda I - T \, \mathrm{is \, not\; Fredholm} \}.$$
\end{definition}

It is well known (see~\cite{At}) that the Fredholm spectrum of a bounded linear operator $T$ coincides with its spectrum in the Calkin algebra.

There are many different definitions of \textbf{essential spectrum} of a linear operator on a Banach space of which Fredholm and semi-Fredholm spectra are but two (though important) examples. We will follow the book~\cite{EE} where five of them are discussed in details.

\begin{definition} \label{d3} (See~\cite[Section I.4, page 40] {EE} Let $T$ be a bounded linear operator on a Banach space $X$. We define the essential spectra of $T$  as the following subsets of $\sigma(T)$.
\begin{itemize}
  \item $\sigma_1(T) = \sigma(T) \setminus \{\zeta \in \mathds{C} : $ the operator $\zeta I - T$ is semi-Fredholm $\}$.
  \item $\sigma_2(T) = \sigma(T) \setminus \{\zeta \in \mathds{C} : $ the operator $\zeta I - T$ is semi-Fredholm and  $nul(\zeta I - T) < \infty \}$.
  \item $\sigma_3(T) = \sigma(T) \setminus \{\zeta \in \mathds{C} : $ the operator $\zeta I - T$ is Fredholm $\}$.
  \item $\sigma_4(T) = \sigma(T) \setminus \{\zeta \in \mathds{C} : $ the operator $\zeta I - T$ is Fredholm and $ind(\zeta I - T) = 0\}$.
  \item $\sigma_5(T) = \sigma(T) \setminus \{\zeta \in \mathds{C} : $ there is a component $C$ of the set $\mathds{C} \setminus \sigma_1(T)$ such that $\zeta \in C$ and the intersection of $C$ with the resolvent set of $T$ is not empty $\}$.
\end{itemize}

\end{definition}

It is well known (see eg.~\cite{EE} or~\cite{Ka}) that the sets $\sigma_i(T), i \in [1, \ldots , 5]$ are nonempty closed subsets of $\sigma(T)$ and that $\sigma_i(T) \subseteq \sigma_j(T), 1 \leq i < j \leq 5$, where all the inclusions can be proper. Nevertheless all the spectral radii
$r_i(T), i =1, \ldots ,5 $ are equal to the same number (see~\cite[Theorem I.4.10]{EE}) which is called the essential spectral radius of $T$.
It is also known (see~\cite{EE}) that the spectra $\sigma_i(T), i=1, \ldots , 4$ are invariant under compact perturbations but $\sigma_5(T)$ in general is not.

The contents of the paper are as follows.

\noindent In Section 2 we describe the semi-Fredholm and Fredholm spectra of weighted composition operators
 $$Tf=w(f \circ \varphi), f \in C(K),\eqno{(1)}$$
 where $C(K)$ is the space of all complex-valued continuous functions on a Hausdorff compact space $K$, $w \in C(K)$, and $\varphi$ is a homeomorphism of $K$ onto itself.

 \noindent In Section 3 we, based on the results of Section 1, describe the spectra $\sigma_i(T), i = 1, \ldots , 5$ where $T$ is an operator of form $(1)$.

 \noindent In Section 4 we consider operators of the form $T = wU: X \to X$ where $X$ is an arbitrary Banach lattice, $w$ is a central operator on $X$, $U$ is a $d$-isomorphism of $X$, and $\sigma(U)$ is a subset of the unit circle. We show that the study of essential spectra of such an operator can be reduced to the study of essential spectra of an appropriate operator of form $(1)$ on $C(K)$ where $K$ is the Stonean compact
 of the Dedekind completion of $X$.

 \noindent In section 5 we touch upon a much more difficult problem of describing essential spectra of weighted compositions induced by \textbf{non invertible} maps and in a special case when $\varphi$ is an open surjection on a compact $K$ provide a criterion for $def(\lambda I - wT_\varphi) = 0$.

 \noindent Finally, in the small appendix we prvide some clarifications about the statement and the proof of Theorem 22 from~\cite{Ki1} which is extensively used in the current paper.

 \section{ When is the operator  $\lambda I - T$  semi-Fredholm?}

We start with recalling some results from~\cite{Ki}

Let $T$ be as in $(1)$. In~\cite[Theorem 3.29 and Theorem 3.31]{Ki} we described two special cases when
$\lambda \in \sigma(T)$ and either $def(\lambda I - T) =0$ or $def(\lambda I - T^\prime) =0$ where $T^\prime$ is the Banach conjugate of $T$. Because these results are crucial for our description of $\sigma_{sf}(T)$ we will reproduce them here, but first we need to recall a couple of notations from~\cite{Ki}.

Let $X$ be a Banach space and $S$ be a bounded linear operator on $X$. We denote the spectrum of $S$ by $\sigma(S)$ and consider the partition of $\sigma(S)$ into two subsets.
$$\sigma_{ap}(S) = \{\lambda \in \sigma(S) : \exists x_n \in X, \; \|x_n\|=1, \; Sx_n - \lambda x_n \mathop \rightarrow \limits_{n \to \infty} 0, \}$$
$$\sigma_r(S) = \sigma(S) \setminus \sigma_{ap}(S).$$

Let $\varphi$ be a homeomorphism of a compact Hausdorff space $K$ onto itself. Then $\varphi^{(0)}$ will mean the identical map of $K$ onto itself, $\varphi^{(m)} = \varphi \circ \varphi^{(m-1)}, m \in \mathds{N}$, and
$\varphi^{(-m)} = (\varphi^{(m)})^{-1}, m \in \mathds{N}$.

Let $m \in \mathds{N}$. We will denote by $\Pi_m$ the subset of $K$ that consists of all $\varphi$-periodic points of period less or equal to $m$.

Let $U$ be a closed subset of $K$ such that $\varphi(U) = U$. We will denote by $T_U$ the operator induced by $(1)$ on the space $C(U)$.

\begin{theorem} \label{t1} (~\cite[Theorem 3.29]{Ki}) Let $\varphi$ be a homeomorphism of the compact space $K$ onto itself, $w \in C(K)$, and $(Tf)(k) = w(k)f(\varphi(k), f \in C(K), k \in K$. Assume that the set
 of all $\varphi$-periodic points is of the first category in $K$. Let $\lambda \in \sigma(T)$. Consider the following statements.

 \noindent $(R)$ The operator $\lambda I - T$ has a right inverse, or equivalently $(\lambda I - T)C(K) = C(K)$), or equivalently $\lambda \in \sigma_r(T^\prime)$.

 \noindent $(L)$  The operator $\lambda I - T$ has a left inverse, or equivalently $\|(\lambda I -T)f\| \geq C\|f\|, \; f \in C(K), C > 0$, or equivalently $\lambda \in \sigma_r(T)$.

 \noindent $(A)$ $K = E \cup Q \cup F$ where the sets $E, Q$, and $F$ are pairwise disjoint, the sets $E$ and $F$ are closed, $\varphi(E) = E$ and $\varphi(F)$ = $F$,
 $\sigma(T_E) \subset \{\xi \in \mathds{C} : |\xi| < |\lambda|\}$, and $\sigma(T_F) \subset \{\xi \in \mathds{C} : |\xi| > |\lambda|\}$.

 \noindent $(B)$
 $$ \forall k \in Q \; \bigcap \limits_{n=0}^\infty cl\{\varphi^{m}(k) : \; m \geq n\} \subseteq F$$
 and
 $$ \forall k \in Q \; \bigcap \limits_{n=0}^\infty cl\{\varphi^{-m}(k) : \; m \geq n\} \subseteq E.$$

 \noindent $(C)$
 $$ \forall k \in Q \; \bigcap \limits_{n=0}^\infty cl\{\varphi^{m}(k) : \; m \geq n\} \subseteq E$$
 and
 $$ \forall k \in Q \; \bigcap \limits_{n=0}^\infty cl\{\varphi^{-m}(k) : \; m \geq n\} \subseteq F.$$

 Then the following equivalencies hold

 \noindent $(1)$ $R \Leftrightarrow A \wedge B$.

 \noindent $(2)$ $L \Leftrightarrow A \wedge C$.

\end{theorem}

\begin{theorem} \label{t2} (~\cite[Theorem 3.31]{Ki}) Let $\varphi$ be a homeomorphism of the compact space $K$ onto itself, $M \in C(K)$, and $(Tf)(k) = M(k)f(\varphi(k), f \in C(K), k \in K$. Let $\lambda \in \sigma(T)$. The following conditions are equivalent.

\noindent $(1)$ $\lambda \in \sigma_r(T^\prime)$ (respectively, $\lambda \in \sigma_r(T))$.

\noindent $(2)$ There are $m \in \mathds{N}$ and an open subset $P$ of $K$ such that  $P \subset \Pi_m$, $\varphi(P) = P$, $\lambda \not \in \sigma(T_{cl P})$, and $K$ can be partitioned as $K = E \cup Q \cup F \cup P$ where the sets
$E,F$, and $Q$ satisfy conditions $A$ and $B$ (respectively $A$ and $C$) of Theorem~\ref{t1}.

\end{theorem}

For any $m \in \mathds{N}$ let us define $w_m \in C(K)$ as
$$w_m = w(w \circ \varphi) \ldots (w \circ \varphi^{(m-1)}).$$
We will need also the following lemma that follows from~\cite[Lemma 3.6 and Theorem 3.12 ]{Ki}

\begin{lemma} \label{l1} Let $T$ be an operator of the form $(1)$ on $C(K)$ and let $\lambda \in \sigma_{ap}(T) \setminus \{0\}$. Then there is a point $k \in K$ such that for every $n \in \mathds{N}$
$$|w_n(k)| \geq |\lambda|^n. \eqno{2(a)}$$
and
$$|w_n(\varphi^{(-n)}(k))| \leq |\lambda|^n.\eqno{2(b)}$$
Moreover, either
\begin{enumerate}[(I)]
  \item $k$ is not a $\varphi$-periodic point,

  \noindent or at least one of the following conditions is satisfied.
  \item $k$ is a $\varphi$-periodic point and for any $n \in \mathds{N}$ and any open neighborhood $V$ of $k$ there is a point $v \in V$ such that either $v$ is not $\varphi$-periodic or its period is greater than $n$.
  \item $k$ is a $\varphi$-periodic point of (the smallest) period $p$ and $w_p(k) = \lambda^p$.
\end{enumerate}

\end{lemma}

Now we can start working on a complete description of Fredholm and semi-Fredholm spectra of operators of form $(1)$.

\begin{lemma} \label{l2} Let $T$ be an operator of the form $(1)$. Assume that $\lambda \in \sigma_{ap}(T) \setminus \{0\}$, that the operator $\lambda I - T$ is semi-Fredholm, and that $nul(\lambda I - T) < \infty$. Let $k \in K$ be a point from the statement of Lemma~\ref{l1}.

Then $k$ is an isolated point of $K$.

\end{lemma}

\begin{proof} Let us first assume that $k$ satisfies condition $(III)$ of Lemma~\ref{l1}. If $k$ is not an isolated point in $K$ then we can find a sequence of points $k_n \in K$ with the properties.
\begin{itemize}
  \item If $m \neq n$ then $\varphi^{(i)}(k_m) \neq \varphi^{(j)}(k_n), 0 \leq i, j \leq p-1$.
  \item $w(\varphi^{(i)}(k_n)) \mathop \rightarrow \limits_{n \to \infty} w(\varphi^{(i)}(k)), \, i=0,1, \ldots , p-1$.
\end{itemize}

Let $u_n$ be the characteristic function of the singleton set $\{k_n\}$. Then $u_n$ can be considered as an element of the second dual
$C^{\prime \prime}(K)$ of norm one. Let $v_n = \sum \limits_{i=0}^{n-1} {\lambda}^{-i} (T^{\prime \prime})^i u_n$. Then it is immediate to see that $\|v_n\| \geq \|u_n\| = 1$ and that $T^{\prime \prime} v_n - \lambda v_n \mathop \rightarrow \limits_{n \to \infty} 0$. Moreover, $v_n$ are pairwise disjoint elements of the Banach lattice $C^{\prime \prime} (K)$ and therefore the sequence $v_n$ cannot contain a norm convergent subsequence. Therefore
(see e.g.~\cite[Corollary I.4.7, page 43]{EE}) the operator $\lambda I - T^{\prime \prime}$ cannot be semi-Fredholm and have a finite dimensional null space. But then~\cite[Theorem I.3.7]{EE} the same is true for the operator $\lambda I - T$ and we come to a contradiction.

Next let us look at the case when $k$ satisfies condition $(II)$ of Lemma~\ref{l1}. In this case it is not difficult to see that we can find points $k_n \in K$ and positive integers $m(n)$ with the properties
\begin{enumerate}[(a)]
  \item $m(n) \mathop \rightarrow \limits_{n \to \infty} \infty$.
  \item For every $n \in \mathds{N}$ all the points $\varphi^{(i)}(k_n), |i| \leq m(n) + 1$, are distinct.
  \item The sets $E_n = \{\varphi^{(i)}(k_n) : \; |i| \leq m(n) + 1 \}, n \in \mathds{N}$ are pairwise disjoint.
  \item For any $n \in \mathds{N}$ the following inequalities hold
  $$|w_i(k_n)| \geq \frac{1}{2} |\lambda|^i, i=1,2, \ldots , m(n) +1$$
  and
  $$|w_i(\varphi^{(-i)}(k_n))| \leq 2 |\lambda|^i, i=1,2, \ldots , m(n) +1$$
\end{enumerate}

Let $u_n$ be the characteristic function of the set $\{\varphi^{(m(n)}(k_n)\}$ and let
$$v_n = \sum \limits_{j=0}^{2m(n)+1} \big{(} 1-\frac{1}{\sqrt{m(n)}} \big{)}^{|j - m(n)|} \overline{\lambda}^j (T^{\prime \prime})^j u_n. $$
Then simple estimates similar to the ones in~\cite{Ki} or~\cite{Ki1} show that
$$\|T^{\prime \prime} v_n - \lambda v_n \| = o(\|v_n\|), n \to \infty .$$
Now we come to a contradiction exactly like on the previous step of the proof.

Finally let us look at the case when the point $k$ is not $\varphi$-periodic. First let us notice that $k \not \in cl \{\varphi^{(i)}(k), i \in \mathds{N}\}$. Indeed, if $k$ were a limit point of the sequence $\varphi^{(i)}(k), k \in \mathds{N}$ then taking into consideration that all the points of this sequence are distinct we can easily produce points $k_n$ from this sequence and positive integers $m(n)$ with properties (a) - (d) above, and come to a contradiction. Similarly we can prove that $k \not \in cl\{\varphi^{(-i)}(k), i \in \mathds{N}\}\}$. Thus $k$ is an isolated point in the closure of its $\varphi$-trajectory. But we assumed that $k$ is not isolated in $K$ and then again simple topological reasons show that we can construct sequences $k_n \in K$ and $m(n) \in \mathds{N}$ with properties (a) - (d). The resulting contradiction ends the proof of the lemma.
\end{proof}

\begin{corollary} \label{c1} Let $T$ be an operator of form $(1)$ on $C(K)$ and let $\lambda \in \mathds{C} \setminus \{0\}$ be such that the operator $\lambda I - T$ is semi-Fredholm and $nul(\lambda I - T) < \infty$. Then either
$$ nul(\lambda I - T) = 0,$$
or there is a finite set $\{k_1, \ldots , k_m\}$ of points isolated in $K$  such that if $F = K \setminus \{\varphi^{(n)}(k_i), n \in \mathds{Z}, i = 1, \ldots , m\}$ then the operator $\lambda I - T_F$ is semi-Fredholm on $C(F)$ and
$$ nul(\lambda I - T_F) = 0. $$
\end{corollary}

\begin{lemma} \label{l3} Let $T$ be an operator on $C(K)$ of form $(1)$, $\lambda \in \mathds{C} \setminus \{0\}$, $\lambda I - T$ be semi-Fredholm, $nul(\lambda I - T) < \infty$, and let $k$ be an isolated not $\varphi$-periodic point of $K$ satisfying inequalities $2(a)$ and $2(b)$ from the statement of Lemma~\ref{l1}. Let
$R = \bigcap \limits_{i=1}^\infty \{\varphi^{(n)}(k), n \geq i \}$ and $L = \bigcap \limits_{i=1}^\infty \{\varphi^{(-n)}(k), n \geq i \}$.

Then
$$\sigma(T_R) \subset \{\xi \in \mathds{C}: \; |\xi| > |\lambda|\}  \eqno{3(a)}$$
and
$$\sigma(T_L) \subset \{\xi \in \mathds{C}: \; |\xi| < |\lambda|\}.   \eqno{3(b)}$$

\end{lemma}

\begin{proof} First notice that $\lambda \not \in \sigma_{ap}(T_R)$. Indeed, otherwise there would be a point $r \in R$ satisfying the inequalities $2(a)$ and $2(b)$. Because $r$ is not an isolated point in $K$ we come to a contradiction with Lemma~\ref{l2}. Therefore there are two possibilities: either $\lambda \in \sigma_r(T_R)$ or $\lambda \not \in \sigma(T_R)$. Let us assume first that $\lambda \in \sigma_r(T_R)$.
Then we can apply Theorem~\ref{t2} to the operator $T_R$. Notice that if $P$ is the subset of $R$ from the statement of Theorem~\ref{t2} the we have $\lambda \Gamma \cap \sigma(T_{cl P}) = \emptyset$. Indeed, otherwise there would be a $\varphi$-periodic point $s \in cl P$ such that
$|w_{p-1}(s)| = |\lambda|^p$ where $p$ is the period of $s$. But because in $K$ the point $s$ is a limit point of the sequence $\varphi^{(n)}(k), n \in \mathds{N}$ we come to a contradiction with Lemma~\ref{l2}. Applying Theorem 3.10 from~\cite{Ki} we see that the set $R$ can be partitioned as $R = E \cup Q \cup F$ where all the sets $E, Q, F$ are nonempty and have the properties $(A)$ and $(C)$ from the statement of Theorem~\ref{t1}. To bring our assumption that $\lambda \in \sigma_r(T_R)$ to a contradiction let $m \in \mathds{N}$ and $R_m = \bigcap \limits_{i=1}^\infty \{\varphi^{(nm)}(k), n \geq i \}$. Obviously $R_m$ is a closed subset of $R$ and $\varphi^{(m)}(R_m) =R_m$. Notice that the set $R_m \cap Q$ is nonempty. Indeed, otherwise we would have $R_m \subset E \cup F$ and because $\varphi (E \cup F) = E \cup F$ it would follow that $R = E \cup F$ in contradiction with our assumption that $Q \neq \emptyset$. It follows from property $(C)$ in the statement of Theorem~\ref{t1} that the sets $R_m \cap E$ and $R_m \cap F$ are nonempty as well. If $m$ is large enough then in some open neighbourhood of $E$ in $K$ we have the inequality $|w_m| \leq |\lambda|^m$ and therefore there is a $p \in \mathds{N}$ such that $|w_m(\varphi^{(pm)}(k)) \leq |\lambda|^m$ and
$|w_m(\varphi^{((p+1)m)}(k)) \geq |\lambda|^m$.  Let $u_m$ be the characteristic function of the singleton $\{\varphi^{((p+2)m)}(k)\}$ and
$$v_m = \sum \limits_{j=0}^{2m} \big{(} 1-\frac{1}{\sqrt{m}} \big{)}^{|j - m|} {\lambda}^{-j} (T^{\prime \prime})^j u_m. $$
Next notice that we can find the sequences $m(i), p(i), i \in \mathds{N}$ such that $\lim_{i \to \infty} m(i) = \infty$, $|w_{m(i)}(\varphi^{(pm(i))})(k)) \leq |\lambda|^{m(i)}$, $|w_{m(i)}(\varphi^{((p+1)m(i))})(k)) \geq |\lambda|^{m(i)}$, and the elements $v_{m(i)}$ are pairwise disjoint in $C^{\prime \prime}(K)$. Finally notice that as in~\cite{Ki} and~\cite{Ki1} we can show that $\|T^{\prime \prime} v_{m(i)} - \lambda v_{m(i)} \| = o(\|v_{m(i)}\|), i \to \infty .$ in contradiction with our assumption that  $\lambda I - T$ is semi-Fredholm and $nul(\lambda I - T) < \infty$.

Thus now we have to consider the case when $\lambda \not \in \sigma(T_R)$. In this case, by Theorem 3.10 from~\cite{Ki} we have $R = E \cup F$ where $E$ and $F$ are disjoint $\varphi$-invariant closed subsets of $R$, $\sigma(T_E) \subset \{\xi \in \mathds{C}: \; |\xi| < |\lambda|\}$
and $\sigma(T_F) \subset \{\xi \in \mathds{C}: \; |\xi| > |\lambda|\}$. The definition of $R$ and elementary topological reasoning show that one of the sets $E$ or $F$ must be empty. If we assume that $F$ is empty we immediately come to a contradiction with the inequality $2(a)$ whence
$R = F$ and $3(a)$ is proved. Statement $3(b)$ can be proved in a similar way.
\end{proof}

Now we can state the first of our main results; before stating it let us notice that the case when the operator $\lambda I - T$ is semi-Fredholm and $nul(\lambda I - T) = 0$ (in other words when $\lambda \in \sigma_r(T)$) is completely described by Theorems~\ref{t1} and~\ref{t2}. We can therefore concentrate on the case when $nul(\lambda I - T) > 0$.

\begin{theorem} \label{t3} Let $T$ be an operator of form $(1)$ on $C(K)$ and let $\lambda \in \sigma(T) \setminus \{0\}$. The following conditions are equivalent.

(I) The operator $\lambda I - T$ is semi-Fredholm and $0 < nul(\lambda I - T) < \infty$.

(II) There is a finite subset $S = \{k_1, \ldots , k_m, s_1, \ldots s_l \}$ of $K$ with the properties

\begin{enumerate}
  \item Every point of $S$ is an isolated point in $K$.
  \item The points $k_i, i=1, \ldots m$, are not $\varphi$-periodic and if the sets $R_i$ and $L_i$ are defined as in the statement of Lemma~\ref{l3} then for each $i \in [1 : m]$ the conditions $3(a)$ and $3(b)$ are satisfied.
  \item The points $s_1, \ldots , s_l$ are $\varphi$-periodic and $\lambda^{p(i)} = w_{p(i)}(s_i)$ where $p(i)$ is the period of the point $s_i$.
  \item $nul(\lambda I - T) = m+l$.
  \item Let $U = \bigcup \limits_{j= - \infty}^\infty \varphi^{(j)}(S)$ and $V = K \setminus U$. Then either $\lambda \not \in \sigma(T_V)$ or
  $\lambda \in \sigma_r(T_V)$.
\end{enumerate}

\end{theorem}

\begin{proof} The implication $(I) \Rightarrow (II)$ follows from Lemma~\ref{l2}, Corollary~\ref{c1}, and Lemma~\ref{l3}.
To prove the implication $(II) \Rightarrow (I)$ notice that by~\cite[Corollary I.4.7, page 43]{EE} it is enough to prove that every sequence
$u_n, n \in \mathds{N}, u_n \in C(K)$ such that  $\|u_n\| = 1$ and $\|Tu_n - \lambda u_n \| \mathop \rightarrow \limits_{n \to \infty} 0$  contains a norm convergent subsequence. Condition $(5)$ guarantees that $\|u_n\|_{C(V)} \mathop \rightarrow \limits_{n \to \infty} 0$. Therefore we can and will assume that $u_n \equiv 0$ on $V$, $n \in \mathds{N}$.
 Next, because the set $S_2 = \{s_1, \ldots , s_l\}$ is a finite and clopen subset of $K$ there is a subsequence of the sequence $u_n$ that converges uniformly on $S_2$ and we can assume without loss of generality that $u_n$ converges uniformly on $S_2$. Thus it remains to prove that if $k \in S$ is not $\varphi$-periodic then we can find a subsequence of $u_n$ that converges uniformly on the set $A = cl\{\varphi^{(i)}(k), i \in \mathds{Z}\}$. If $\|u_n\|_{C(A)} \mathop \rightarrow \limits_{n \to \infty} 0$ then it is nothing to prove. Therefore we can assume without loss of generality that $\|u_n\|_{C(A)} = 1$.  Let $a_n$ be such a point in $A$ that $|u_n(a_n)| = 1$ and let $a$ be a limit point of the set $\{a_n\}$. The point $a$ cannot belong to the set $V$. Indeed at that point we have inequalities $2(a)$ and $2(b)$. On the other hand if $a \in V$ then $a \in cl A \setminus A$ in contradiction with condition $(2)$. Thus $a \in \{\varphi^{(i)}(k), i \in \mathds{Z} \}$ and we can assume without loss of generality that $a = k$ and that the sequence $u_n(a)$ converges to $1$. We define the function $u$ on $K$ in the following way

 \begin{equation*}
u(x) =
\begin{cases} 1     , & \text{if $x = k$,}
\\
 \lambda^i/w_i(k) , & \text{if $x = \varphi^{(i)}(k), i \in \mathds{N}$,}
\\
w_i(\varphi^{(-i)}(k)/\lambda^i , & \text{if $x = \varphi^{(-i)}(k), i \in \mathds{N}$,}
\\
 0, &\text{ otherwise.}
\end{cases}
\end{equation*}
Condition $(2)$ guarantees that $u \in C(K)$. Clearly the sequence $u_n$ converges to $u$ pointwise on $K$ and we will prove now that some subsequence of it converges to $u$ uniformly. Indeed, otherwise we can find a positive constant $c$ and a strictly increasing sequence of positive integers $i(n)$ such that either $|u_n(\varphi^{(i(n))})(k)| \geq c$ or $|u_n(\varphi^{(-i(n))})(k)| \geq c$. We will assume the first case because the second one can be considered absolutely similarly. Keeping in mind that $u_n \equiv 0$ on $V$ and that $u_n$ converge to $0$ pointwise on $K$ we can find another strictly increasing sequence of integers $j(n)$ such that $i(n) - j(n) \mathop \rightarrow \limits_{n \to \infty} \infty$ and $|u_n(\varphi^{(i(n)\pm j(n))})| \leq c/n$. Let $v_n$ be the restriction of $u_n$ on the set $\{\varphi^{(d)}(k),
d \in [i(n) - j(n) : i(n) + j(n)]\}$. Then $\|v_n\| \geq c$ and $\|Tv_n - \lambda v_n \| \mathop \rightarrow \limits_{n \to \infty} 0$. Let $r_n$ be a point in $K$ where $|v_n|$ takes its maximum value and $r$ be a limit point of the sequence $r_n$. Then $r \in cl A \setminus A$ and at point $r$ we have inequalities $2(a)$ and $2(b)$ in contradiction with condition $(2)$ of the current theorem.

\end{proof}

Our next goal is to describe when the operator $\lambda I - T$, $\lambda \in \mathds{C} \setminus \{0\}$, is semi-Fredholm and $def(\lambda I - T) < \infty$. We start with the following remark.

 \begin{remark} \label{r1} The case when $def(\lambda I - T) = 0$ was discussed in Theorems~\ref{t1} and~\ref{t2}, and therefore we will assume that
$0 < def(\lambda I - T) < \infty$. In this case the operator $\lambda I - T^\prime$ is semi-Fredholm and $nul(\lambda I - T^\prime) = def(\lambda I - T)$. (See e.g.~\cite[Theorem I.3.7, page 29]{EE}).  Consider now an auxiliary operator
$$(\tilde{T}f)(x) = w(x)f(\varphi^{(-1)}(x)), \; f \in C(K), \; x \in K. \eqno{(\bigstar)}$$
Then (see~\cite{Ki} and~\cite{Ki1}) $\sigma(\tilde{T}) = \sigma(T)$ whence $\lambda \in \sigma(\tilde{T})$. Assume that $\lambda \in \sigma_{ap}(\tilde{T})$. Then (see~\cite{Ki1}) there is a point $k \in K$ such that for every $n \in \mathds{N}$
$$|w_n(k)| \leq |\lambda|^n. \eqno{4(a)}$$
and
$$|w_n(\varphi^{(-n)}(k))| \geq |\lambda|^n.\eqno{4(b)}$$
 Notice now that $k$ must be an isolated point in $K$. Indeed, otherwise similar to the proof of Lemma~\ref{l1} we can construct a sequence of Borel regular measures $\mu_n$ on $K$ (actually every $\mu_n$ is a finite linear combination of Dirac measures) such that the measures $\mu_n$ are pairwise disjoint in $C(K)^\prime$ and $\|T^\prime \mu_n - \lambda \mu_n \| = o(\|\mu_n\|), n \rightarrow \infty$ in contradiction with our assumption that $\lambda I - T^\prime$ is semi-Fredholm and $nul(\lambda I - T^\prime) < \infty$.

\end{remark}

\begin{corollary} \label{c2} Let $T$ be an operator of form $(1)$ on $C(K)$ and let $\lambda \in \mathds{C} \setminus \{0\}$ be such that the operator $\lambda I - T$ is semi-Fredholm and $0 < def(\lambda I - T) < \infty$. Then there is a finite set $S = \{k_1, \ldots , k_m, s_1, \ldots , s_l\}$ of points isolated in $K$  such that the points $k_i, i = 1, \ldots , m$ are not $\varphi$-periodic and satisfy conditions $4(a)$ and $4(b)$. Points $s_1, \ldots , s_l$ are $\varphi$-periodic and $\lambda^{p(i)} = w_i(s_i), i = 1, \ldots s$, where $p(i)$ is the period of the point $s_i$. Moreover, if $F = K \setminus \bigcup \limits_{i= - \infty}^\infty \varphi^{(i)}(S)$, then the operator $\lambda I - T_F$ is semi-Fredholm on $C(F)$ and $ def(\lambda I - T_F) = 0 $.

\end{corollary}

\begin{proof} The proof follows immediately from Remark~\ref{r1} and Theorems~\ref{t1} and~\ref{t2}.

\end{proof}

\begin{lemma} \label{l4} Assume conditions of Corollary~\ref{c2} and let $S$ be the set from the statement of that corollary. Assume also that
$k \in S$ is not $\varphi$-periodic. Like in Lemma~\ref{l3} let $R = \bigcap \limits_{i=1}^\infty \{\varphi^{(n)}(k), n \geq i \}$ and $L = \bigcap \limits_{i=1}^\infty \{\varphi^{(-n)}(k), n \geq i \}$.

Then
$$\sigma(T_R) \subset \{\xi \in \mathds{C}: \; |\xi| < |\lambda|\}  \eqno{5(a)}$$
and
$$\sigma(T_L) \subset \{\xi \in \mathds{C}: \; |\xi| > |\lambda|\}.   \eqno{5(b)}$$

\end{lemma}

\begin{proof} It follows from Remark~\ref{r1} that $\lambda \not \in \sigma_{ap}(\tilde{T}_{R \cup L})$ where $\tilde{T}$ is defined
 by the equation $(\bigstar)$. The rest of the proof goes very similar to the proof of Lemma~\ref{l3} and can be omitted.

\end{proof}

\begin{theorem} \label{t4} Let $T$ be an operator of form $(1)$ on $C(K)$ and let $\lambda \in \sigma(T) \setminus \{0\}$. The following conditions are equivalent.

(I) The operator $\lambda I - T$ is semi-Fredholm and $0 < def(\lambda I - T) < \infty$.

(II) There is a finite subset $S = \{k_1, \ldots , k_m, s_1, \ldots s_l \}$ of $K$ with the properties

\begin{enumerate}
  \item Every point of $S$ is an isolated point in $K$.
  \item The points $k_i, i=1, \ldots m$, are not $\varphi$-periodic and if the sets $R_i$ and $L_i$ are defined as in the statement of Lemma~\ref{l3} then for each $i \in [1 : m]$ the conditions $5(a)$ and $5(b)$ are satisfied.
  \item The points $s_1, \ldots , s_l$ are $\varphi$-periodic and $\lambda^{p(i)} = w_{p(i)}(s_i)$ where $p(i)$ is the period of the point $s_i$.
  \item $def(\lambda I - T) = m+l$.
  \item Let $U = \bigcup \limits_{j= - \infty}^\infty \varphi^{(j)}(S)$ and $V = K \setminus U$. Then either $\lambda \not \in \sigma(T_V)$ or
  $\lambda \in \sigma_r(T_V^\prime)$.
\end{enumerate}

\end{theorem}

\begin{proof} The implication $(I) \Rightarrow (II)$ follows from Remark~\ref{r1}, Corollary~\ref{c2}, and Lemma~\ref{l4}.

To prove the implication $(II) \Rightarrow (I)$ we have to prove that if $(II)$ is satisfied and $\mu_n$ is a sequence of regular Borel measures on $K$ such that $\|\mu_n\|=1$ and $T^\prime \mu_n - \lambda \mu_n \mathop \rightarrow \limits_{n \to \infty} 0$ then the sequence $\mu_n$ contains a norm convergent subsequence. It follows from $II(5)$ that without loss of generality we can assume that there is $k \in [k_1, \ldots , k_m]$ such that $supp \, \mu_n \subset \{\varphi^i(k) : \; i \in \mathds{Z} \}, n \in \mathds{N}$. By the well known criterion of compactness in $l^1$ (see e.g.~\cite{DS}) it is enough to prove that for any positive $\varepsilon$ there is an $m = m(\epsilon) \in \mathds{N}$ such that
$|\mu_n|(\{\varphi^{(i)}(k), \; |i| > m \}) < \varepsilon , n \in \mathds{N} $. Assume to the contrary that there is a positive $\varepsilon$ and a subsequence $\nu_s = \mu_{n_s}$ of the sequence $\mu_n$ such that for any $l \in \mathds{N}$ we have $|\nu_s|(\{\varphi^{(i)}(k) : |i| > l \}) \geq \varepsilon , s \in \mathds{N}$. The operator $T$ is the product of a central operator and a $d$-isomorphism and therefore the conjugate operator $T^\prime$ preserves disjointness. Thus $|T^\prime||\nu_s| - |\lambda||\nu_s| \mathop \rightarrow \limits_{s \to \infty} 0$. Let $\tau$ be a limit point of the sequence $|\nu_s|$ in the weak-$\star$ topology on $C(K)^\prime$. Then $\tau$ is a probability measure on $K$ and
$|T^\prime| \tau = |\lambda| \tau$. Let $Tr(k) = \{\varphi^{(i)}(k), i \in \mathds{Z} \}$. Then clearly $\tau(cl Tr(k) \setminus Tr(k)) \geq \varepsilon$. Therefore $|\lambda| \in \sigma_{ap}(T^\prime , C( cl Tr(k) \setminus Tr(k))$. It follows from Theorems~\ref{t1} and~\ref{t2} that there is a point $l \in cl Tr(k) \setminus Tr(k)$ such that at point $l$ we have the inequalities $5(a)$ and $5(b)$. Because $l$ is a limit point of the set $Tr(k)$ we easily conclude that $\lambda \in \sigma_{ap}(T^\prime)$, a contradiction.

\end{proof}

Now we can answer the question when operator $\lambda I - T$ ($\lambda \neq 0$) is Fredholm. The next theorem follows directly from Theorems~\ref{t3} and~\ref{t4}.

\begin{theorem} \label{t5} Let $T$ be an operator of form $(1)$ on $C(K)$ and $\lambda \in \sigma(T) \setminus \{0\}$. The following conditions are equivalent.

$(I)$ The operator $\lambda I - T$ is Fredholm.

$(II)$ There is a finite subset $S = \{k_1, \ldots, k_m, l_1, \ldots , l_n, s_1, \ldots , s_q \}$ of $K$ such that

\begin{enumerate}[(a)]
      \item Every point $k_i, i \in [1, \ldots , m]$ is not $\varphi$-periodic and satisfies conditions $2(a), 2(b), 3(a)$, and $3(b)$.
      \item Every point $l_i, i \in [1, \ldots , n]$ is not $\varphi$-periodic and satisfies conditions $4(a), 4(b), 5(a)$, and $5(b)$.
      \item Every point $s_i, i \in [1, \ldots , q]$ is $\varphi$-periodic and $\lambda^{p(i)} = w_{p(i)}$ where $p(i)$ is the period of the point $s_i$.
      \item Let $V = \bigcup \limits_{i=-\infty}^\infty \varphi^{(i)}(S)$. Then $\lambda \not \in \sigma(T_{K \setminus V})$.

    \end{enumerate}

    Moreover, if condition $(II)$ is satisfied then $ind(\lambda I - T) = n - m$.

\end{theorem}

\begin{remark} \label{r2} It follows from Theorems 3.10 and 3.12 in~\cite{Ki} that condition $II(d)$ in the statement of Theorem~\ref{t5} is equivalent to the following. The set $K \setminus V$ is the union of three disjoint subsets (of which two might be empty) $E$, $F$, and $P$ with the properties.
\begin{itemize}
  \item The set $E$ is closed in $K$, $\varphi(E) = E$, and if $E \neq \emptyset$ then $\sigma(T_E) \subset \{\zeta \in \mathds{C} : |\zeta| > |\lambda| \}$.
  \item The set $F$ is closed in $K$, $\varphi(F) = F$, and if $F \neq \emptyset$ then $\sigma(T_F) \subset \{\zeta \in \mathds{C} : |\zeta| < |\lambda| \}$.
  \item The set $P$ is open in $K$, $\varphi(P) = P$, and if $P \neq \emptyset$ then $P$ consists of $\varphi$-periodic points with periods bounded by some $N \in \mathds{N}$ and $\lambda \not \in \sigma(T_{cl P})$.
\end{itemize}

\end{remark}

To complete our description of Fredholm and semi-Fredholm operators of the form $\lambda I - T$ it remains to consider the case when $\lambda = 0$. Of course the only interesting case is $0 \in \sigma(T)$ and therefore we assume that the set $Z = \{k \in K : w(k) = 0\}$ is not empty.
It follows immediately from Corollary I.4.7 in~\cite{EE} that if $T$ is semi-Fredholm then $Z$ must be a clopen subset of $K$. Moreover, clearly $Z$ must be a finite subset of $K$. Thus we obtain the following simple proposition which is most probably known.

\begin{proposition} \label{p1} Let $T$ be an operator of form $(1)$ on $C(K)$. The following conditions are equivalent.
\begin{enumerate}
  \item $T$ is semi-Fredholm.
  \item $T$ is Fredholm.
  \item $T$ is Fredholm and $ind\,  T =0$.
  \item The set $Z = \{k \in K : w(k) = 0\}$ is finite and consists of points isolated in $K$.
\end{enumerate}

\end{proposition}

\section{Essential spectra of operators of form $(1)$.}

We are going now to describe the sets $\sigma_i(T), i=1, \ldots, 5$ for operators of form $(1)$. To avoid unnecessary complicated and cumbersome statements we will start with the following remark.

\begin{remark} \label{r3} (1) In view of Proposition~\ref{p1} it is enough to describe the sets $\sigma_i(T) \setminus \{0\}$.

(2) For any $p \in \mathds{N}$ let $\Pi^p$ be the set of all \textbf{non-isolated} $\varphi$-periodic points of period $p$. Then it follows from the results of the previous section that
$$cl \{\lambda \in \mathds{C} : \; \exists p \in \mathds{N}, \exists k \in Int(\Pi^p) \; \lambda^p = w_p(k) \} \subseteq \sigma_1(T).$$
Therefore in our next theorem we will assume that $Int(\Pi^p) = \emptyset , p \in \mathds{N}$.
\end{remark}

From the results in the previous section we obtain the following theorem.

\begin{theorem} \label{t6}
Let $T$ be an operator of form $(1)$ on $C(K)$.  Assume that $Int(\Pi^p) = \emptyset , p \in \mathds{N}$ where $\Pi^p$ are the sets introduced in Remark~\ref{r3}  Then the sets $\sigma_i(T), i = 1, \ldots , 5$ are rotation invariant. Moreover, if  $\lambda \in \mathds{C} \setminus{0}$ then

\begin{enumerate}
  \item The following conditions are equivalent.

  $(IA)$ $\lambda \in \sigma_1(T)$.

  $(IB)$ There are two (not necessarily distinct) points $k_1, k_2 \in K$ such that \textbf{none of them is an isolated point} in $K$, at $k_1$ we have inequalities $2(a)$ and $2(b)$, and at $k_2$ - inequalities $4(a)$ and $4(b)$.

  \item  The following conditions are equivalent.

  $(IIA)$ $\lambda \in \sigma_2(T)$.

  $(IIB)$ There is a \textbf{non-isolated point} $k \in K$ satisfying inequalities $2(a)$ and $2(b)$.

  \item The following conditions are equivalent.

  $(IIIA)$ $\lambda \in \sigma_3(T)$.

  $(IIIB)$ There is a \textbf{non-isolated point} $k \in K$ such that at this point either inequalities $2(a)$ and $2(b)$ or inequalities $4(a)$ and $4(b)$ are satisfied.

  \item The following conditions are equivalent.

  $(IVA)$ $\lambda \in \sigma_4(T)$.

  $(IVB)$ Let $K_1$ (respectively $K_2$) be the set of all points in $K$ that are not isolated $\varphi$-periodic points and satisfy $2(a)$ and $2(b)$ (respectively $4(a)$ and $4(b)$). Then either at least one of these sets is infinite or they have distinct cardinalities.

  \item The following conditions are equivalent.

  $(VA)$ $\lambda \in \sigma_5(T)$

  $(VB)$ There is a $k \in K$ such that $k$ is not an isolated $\varphi$-periodic point and such that at this point either inequalities $2(a)$ and $2(b)$ or inequalities $4(a)$ and $4(b)$ are satisfied.
\end{enumerate}
\end{theorem}

The results of the previous section together with Theorems~\ref{t1} and~\ref{t2} provide the following corollary.

\begin{corollary} \label{c3} Let $T$ be an operator of form $(1)$ on $C(K)$. Then

(1) If $T$ is band irreducible then $\sigma_1(T) = \sigma(T)$.

(2) If $K$ has no $\varphi$-periodic isolated points then $\sigma_5(T) \setminus \{0\} = \sigma(T) \setminus \{0\}$.

(3) If $K$ has no isolated points then $\sigma_3(T) = \sigma(T)$.

(4) Let $O$ be the set of all isolated $\varphi$-periodic points in $K$. Then the essential spectral radius $\rho_e(T)$ of $T$ can be computed in the following way (see e.g.~\cite[Theorem 3.23]{Ki}).

$$\rho_e(T) = \rho(T_{K \setminus O}) = \max \limits_{\mu \in \mathcal{M}_\varphi} \exp \int{\ln |w| d \mu}, $$
where $\mathcal{M}_\varphi$ is the set of all $\varphi$-invariant regular Borel probability measures on $K \setminus O$. In particular, if $O = \emptyset$ then $\rho_e(T) = \rho(T)$.
\end{corollary}

\newpage

The following diagrams illustrate our results from Theorem~\ref{t6}

\centerline{ \textbf{ DIAGRAM I}}

\includegraphics[scale=0.50]{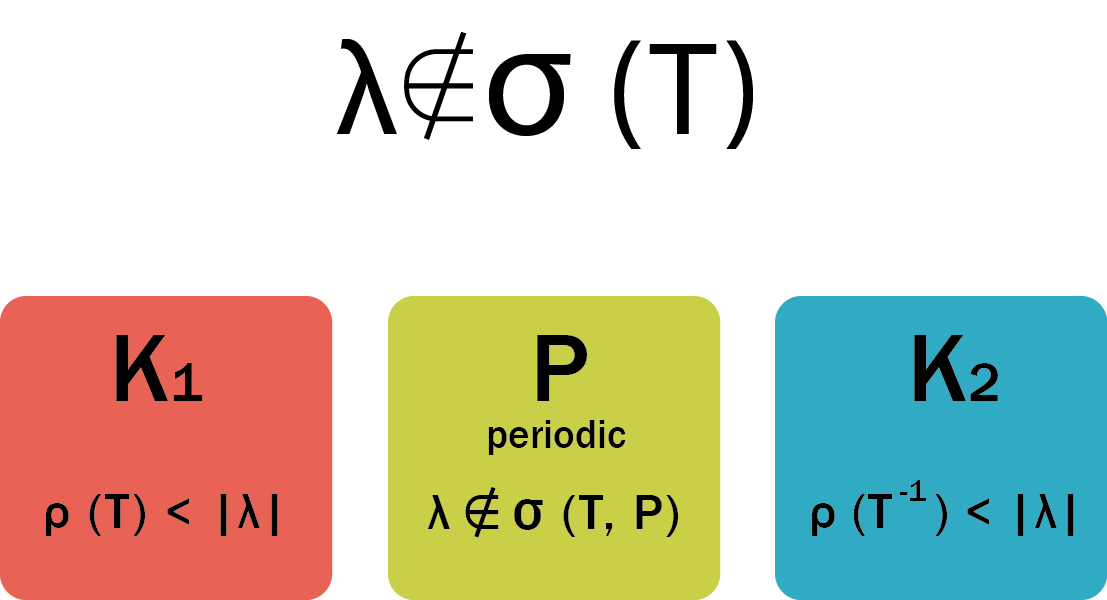}

\newpage
\centerline{ \textbf{ DIAGRAM II}}
\includegraphics[scale=0.50]{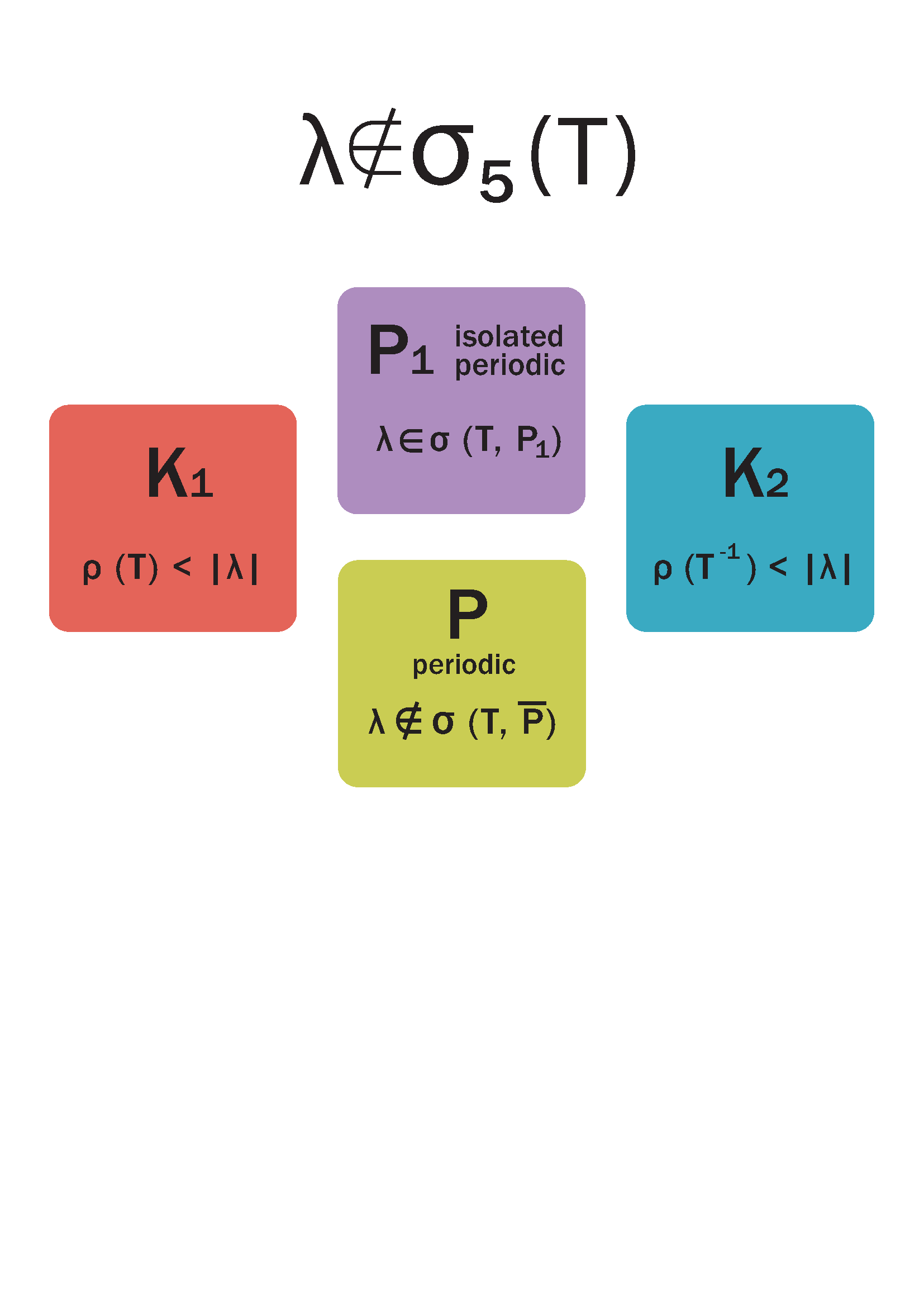}

\newpage
\centerline{ \textbf{ DIAGRAM III}}

\includegraphics[scale=0.20]{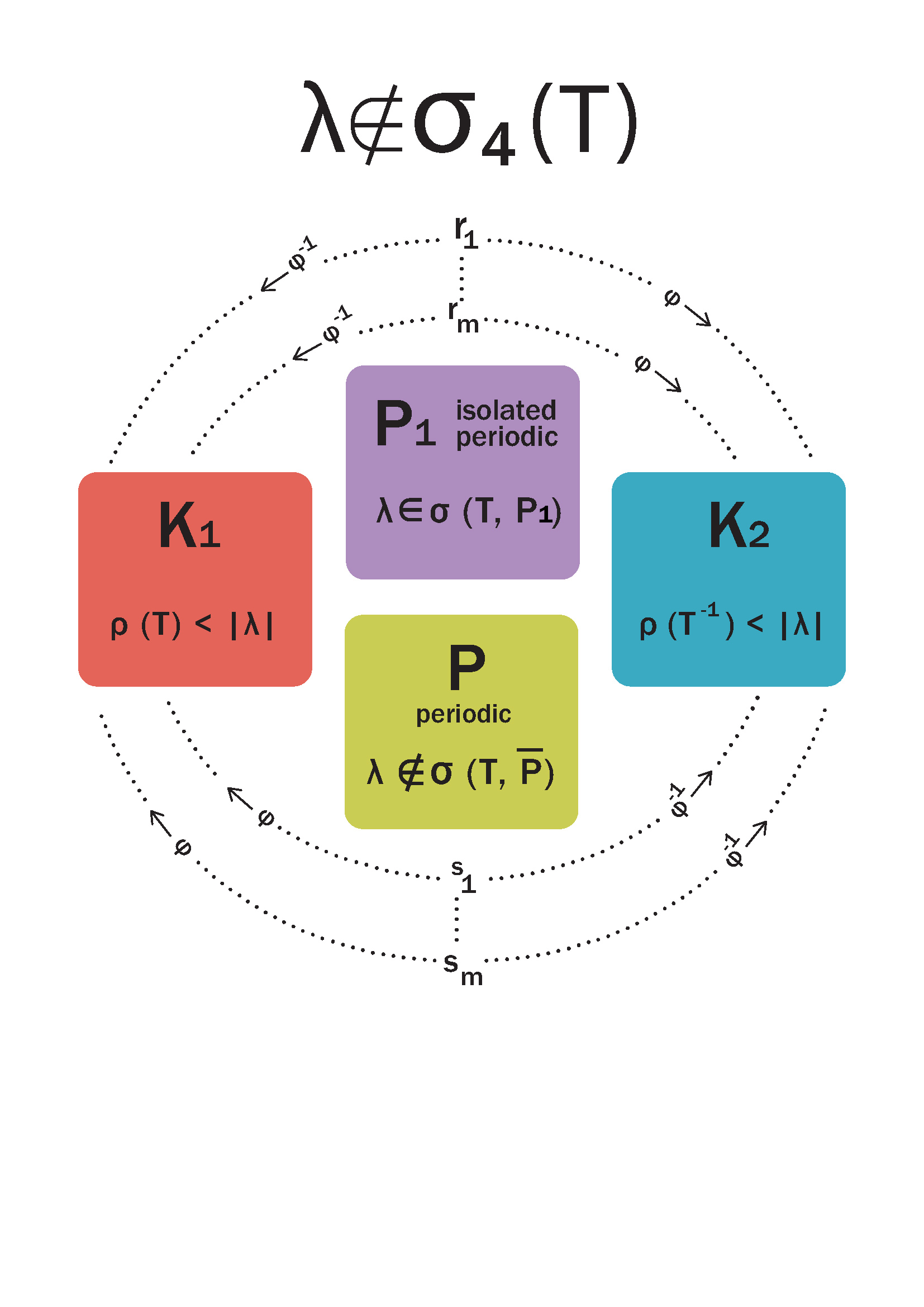}

\newpage
\centerline{ \textbf{ DIAGRAM IV}}

\includegraphics[scale=0.20]{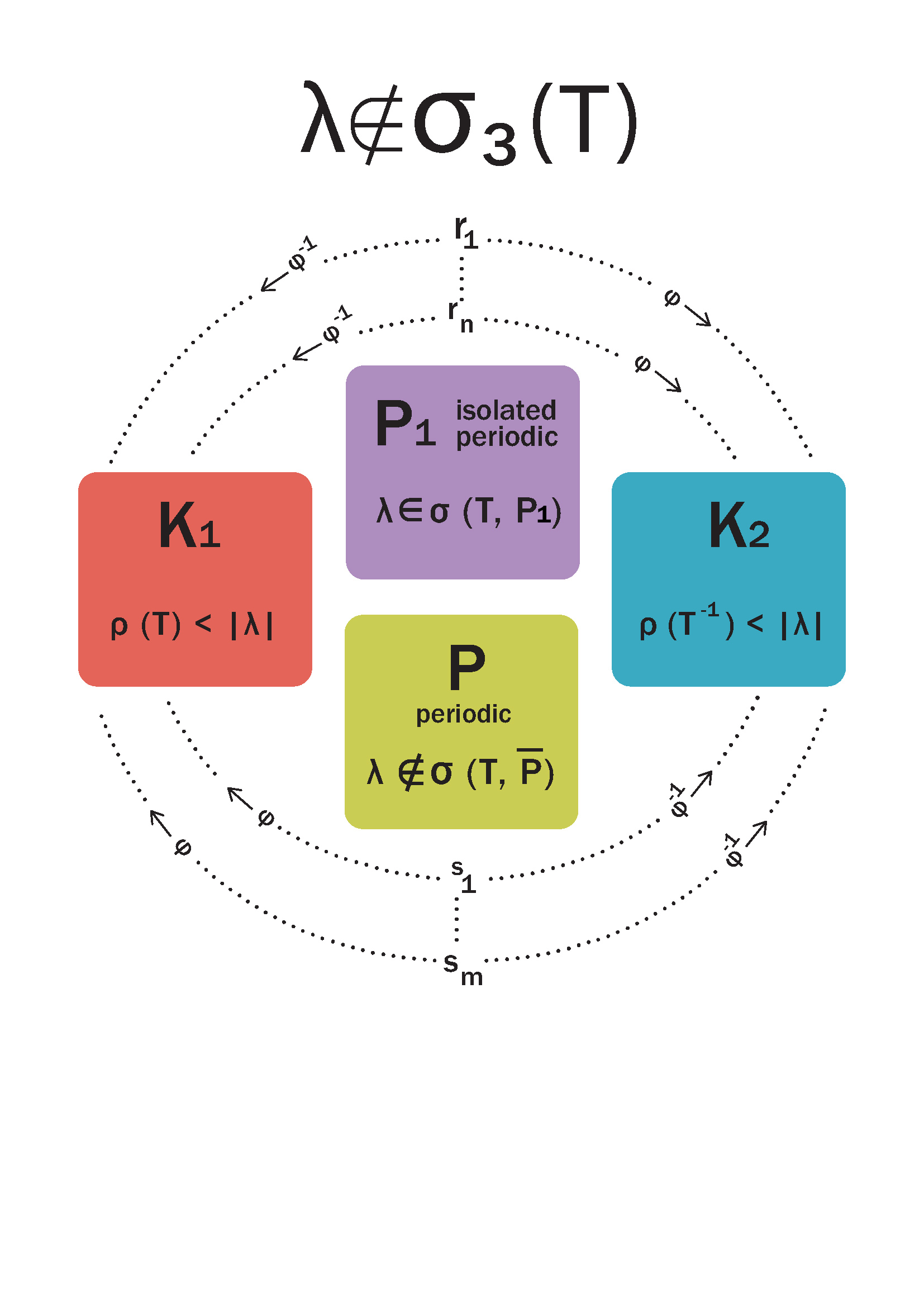}

\newpage
\centerline{ \textbf{ DIAGRAM V}}

\includegraphics[scale=0.20]{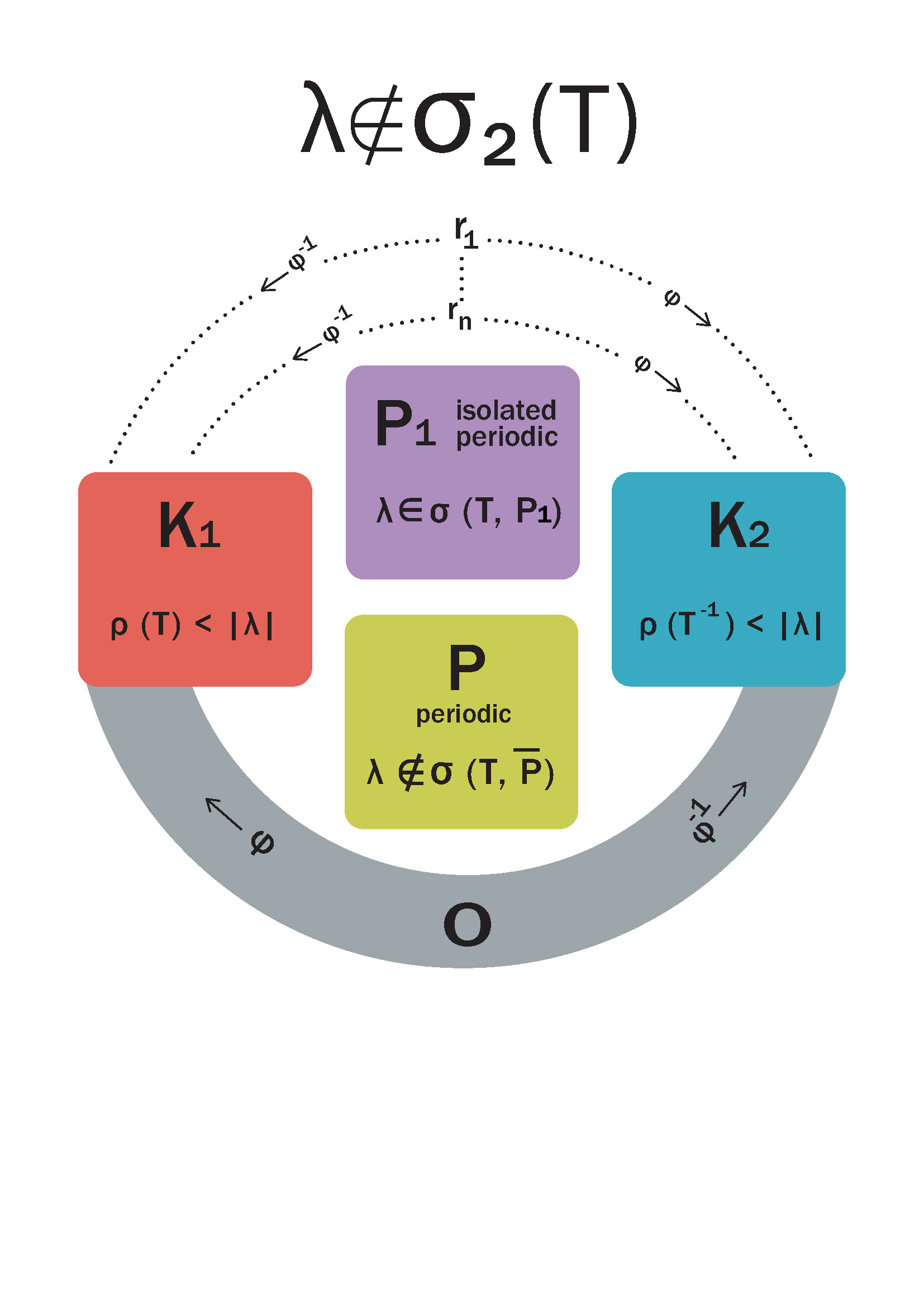}

\newpage
\centerline{ \textbf{ DIAGRAM VI}}

\includegraphics[scale=0.20]{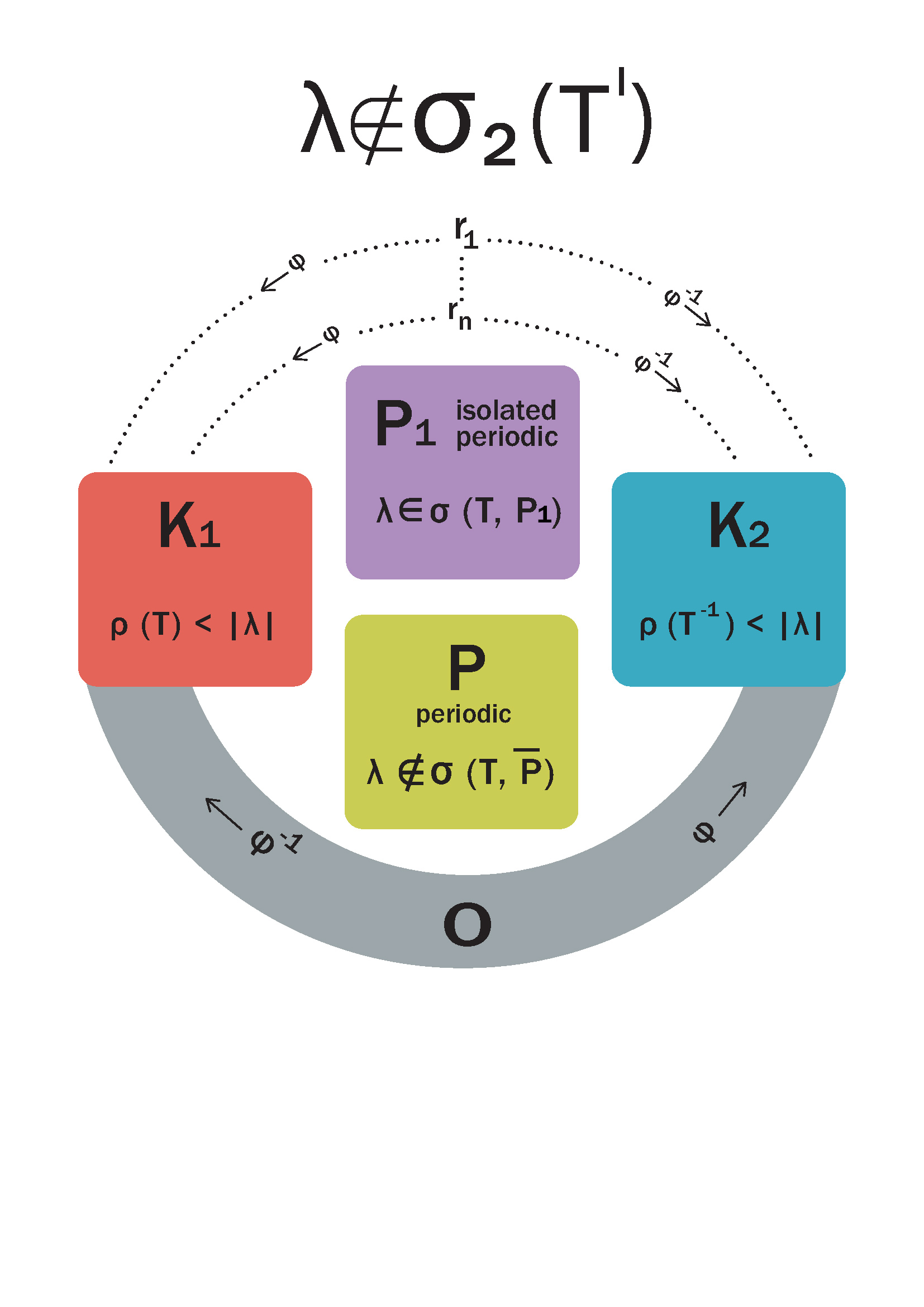}

Our next Theorem (that also follows from the results in \cite{Ki} and in the previous section) provides an alternative description of the sets $\sigma_1(T), \sigma_2(T), \sigma_3(T)$, and $\sigma_5(T)$ that complements the one in Theorem~\ref{t6}. To state it let us introduce or recall the following notations.

\begin{itemize}
  \item $\Gamma$ is the unit circle. $\Gamma = \{\zeta \in \mathds{C} : |\zeta| = 1\} $.
  \item $O_1$ is the set of all isolated $\varphi$-periodic points in $K$.
  \item $K_1 = K \setminus O_1$.
  \item $\Pi^p, p \in \mathds{N}$ is the set of all $\varphi$-periodic points of period $p$ in $K_1$.
  \item $\Sigma = cl \{\lambda \in \mathds{C} : \; \exists p \in \mathds{N}, \exists k \in Int(\Pi^p) \; \lambda^p = w_p(k) \}$.
  \item $O_2 = \bigcup \limits_{p=1}^\infty Int(\Pi^p)$.
  \item $K_2 = K_1 \setminus O_2 $.
  \item $O_3$ is the set of all isolated points in $K$ satisfying either conditions $3(a)$ and $3(b)$ or $5(a)$ and $5(b)$.
  \item $K_3 = K_2 \setminus O_3$.
  \item $O_4$ is the set of all points in $K$ with the following property. If $k \in O_4$ then there is an open neighborhood $V$ of $k$ such that the sets $cl \, \varphi^{(i)}(V), i \in \mathds{Z}$, are pairwise disjoint and $\rho(T_R) < 1/\rho(T^{-1}_L)$ where
      $L = \bigcap \limits_{n=1}^\infty cl \bigcup \limits_{i=n}^\infty \varphi^{(-i)}(V)$ and  $R = \bigcap \limits_{n=1}^\infty cl \bigcup \limits_{i=n}^\infty \varphi^{(i)}(V)$.
  \item $K_4 = K_3 \setminus O_4$.
  \item $O_5$ is defined similarly to $O_4$ but we require the inequality $\rho(T_L) < 1/\rho(T^{-1}_R)$.
  \item $K_5 = K_3 \setminus O_5$.
\end{itemize}

Let us also agree that if $\sigma$ is a subset of $\mathds{C}$ then $\sigma \Gamma = \{\lambda \gamma : \lambda \in \sigma, \gamma \in \Gamma\}$.

\begin{theorem} \label{t7} Let $T$ be an operator on $C(K)$ of form $(1)$. Then
\begin{itemize}
  \item $$\sigma_5(T) = \sigma(T_{K_2}) \Gamma \cup \Sigma. $$
  \item $$\sigma_3(T) = \sigma(T_{K_3}) \Gamma \cup \Sigma. $$
  \item $$\sigma_2(T) = \sigma(T_{K_4}) \Gamma \cup \Sigma $$.
  \item $$\sigma_1(T) = (\sigma(T_{K_4})\Gamma \cap \sigma(T_{K_5})\Gamma) \cup \Sigma$$.
\end{itemize}

\end{theorem}

\section{ Essential spectra of weighted $d$-isomorphisms of Banach lattices.}

In this section we will consider the essential spectra of an operator $T$ on a Banach lattice $X$ that allows the following representation.
$$ T = wU, \; w \in Z(X), \sigma(U) \subset \Gamma. \eqno(6)$$
\begin{remark} The question when every invertible disjointness preserving operator on a Banach lattice $X$ can be represented in form $(6)$ is of independent interest. In the case when $X$ is a $C(K)$ or an $L^p$-space , $1 \leq p \leq \infty$, the answer is positive and moreover $U$ can be chosen as an invertible isometry of $X$. The answer is also positive in the case when $X$ is a Banach function space and its upper and lower Boyd indices are both equal to $p$, $1 < p < \infty$, though in this case in general we cannot claim that $U$ is an isometry of $X$.
\end{remark}

We will assume first that $X$ is a Dedekind complete Banach lattice. Let $K$ be the Stonean compact space of $X$; then the center $Z(X)$ can be identified with $C(K)$. The map $f \rightarrow UfU^{-1}, \; f \in C(K)$ defines an isomorphism of the algebra $C(K)$. Let $\varphi$ be the corresponding homeomorphism of $K$ onto itself. We consider the weighted composition operator $S$ on $C(K)$ defined as
$$(Sf)(k) = w(k)f(\varphi(k)), \; f \in C(K), \; k \in K. \eqno{(7)}$$
Our nearest goal is to prove the following result.

\begin{theorem} \label{t8} Let $X$ be a Dedekind complete Banach lattice and $Z(X) \sim C(K)$ be the center of $X$. Let $T$ be an operator on $X$ defined by formula $(6)$ and $S$ be the corresponding operator on $C(K)$ defined by $(7)$. Then the essential spectra of operators $T$ and $S$ coincide.
$$\sigma_i(T) = \sigma_i(S), \; i=1, \ldots , 5 . $$

\end{theorem}

\begin{proof} $(1)$ In this part we will prove that $0 \in \sigma_i(T) \Leftrightarrow 0 \in \sigma_i(S), i = 1, \ldots , 5$. Assume that $0 \in \sigma(T)$ then obviously $0 \in \sigma(S)$. Assume that $0 \not \in \sigma_1(S)$, then by Proposition~\ref{p1} the set $Z(w)$ of zeros of $w$ is finite and consists of points isolated in $K$. Let $\chi$ be the characteristic function of the set $Z(w)$, then the operator $\chi U$ is finite dimensional while the operator $T + \chi U$ is invertible on $X$. Therefore $T$ is a Fredholm operator. Notice that $null(T) = null(S) = card(Z(w))$. The operator $w^\prime$ is a central operator on the Banach dual $X^\prime$ of $X$ and because the operator $T^\prime + U^\prime \chi^\prime$ is invertible on $X^\prime$ we see that $null(T^\prime) = card(Z(w^\prime)) \leq card(Z(w)) = null(T)$. On the other hand a similar reasoning shows that $null(T) = null(T^{\prime \prime}) \leq null(T^\prime)$ whence $ind(T) = 0$.

Now assume that $0 \in \sigma_1(S)$. By proposition~\ref{p1} the set $Z(w)$ contains a point that is not isolated in $K$. Then we can construct a sequence of pairwise disjoint elements $x_n \in X$ such that $\|x_n\| = 1$ and $\chi_n x_n = x_n$ where $\chi_n$ is the characteristic function of the set $\{k \in K : \; |w(k)| \leq 1/n \}$. Then it is immediate to see that $Tx_n \mathop \rightarrow \limits_{n \to \infty} 0$ whence $0 \in \sigma_2(T)$. But the set $Z(w^\prime)$  must also contain a point that is not isolated in $K^\prime$ where $K^\prime$ is the Gelfand compact of the ideal center of $X^\prime$. Indeed, otherwise the sum of $T^\prime$ and a finite dimensional operator would be invertible and $T$ (together with $T^\prime$) would be Fredholm in contradiction with our assumption. Thus $0 \in \sigma_2(T^\prime)$ whence $0 \in \sigma_1(T)$.

It follows from the two previous paragraphs that
$$ 0 \in \sigma_i(T) \Leftrightarrow 0 \in \sigma_i(S), i = 1,2,3,4.  \eqno{(8)}$$
The equivalence $0 \in \sigma_5(T) \Leftrightarrow 0 \in \sigma_5(S)$ follows from Definition~\ref{d3}, from $(8)$, and from the equality
$\sigma(T) = \sigma(S)$ proved in ~\cite[Theorem 22]{Ki1}

$(2)$ Here we will prove that $\sigma_1(S) \subseteq \sigma_1(T)$. Let $\lambda \in \sigma_1(S) \setminus \{0\}$. Without loss of generality we can assume that $\lambda = 1$. Then it follows from Theorems~\ref{t6} and~\ref{t7} as well as from the Frolik's theorem~\cite[Theorem 6.25, p. 150]{Wa} \footnote{The Frolik's theorem states that if $K$ is an extremally disconnected compact space and $\varphi$ is a homeomorphism of $K$ into itself then the set of all fixed points of $\varphi$ is clopen in $K$.} that we have to consider two possibilities.

\noindent $(2a)$ There are a point $k \in K$ and $p \in \mathds{N}$ such that $k \in Int(\Pi^p)$ and $ w_p(k)= 1$. We can find clopen nonempty subsets $E_n, n \in \mathds{N}$ of $Int(\Pi^p)$ such that
\begin{itemize}
  \item $\varphi^{(i)}(E_n) \cap E_n = \emptyset$, $n \in \mathds{N}$, $1 \leq i \leq p-1 $.
  \item $\varphi^{(p)}(E_n) = E_n, n \in \mathds{N}$.
  \item If $m \neq n$ then $\varphi^{(i)} (E_m) \cap \varphi^{(j)} (E_n) = \emptyset, 0 \leq i,j \leq p-1$.
  \item $\max \limits_{k \in \bigcup \limits_{i=0}^{p-1} \varphi^{(i)}(E_n)} |w_p(k) - 1| < 1/n , n \in \mathds{N}$.
\end{itemize}
 Let $P_n$ be the band projection on $X$ corresponding to the set $E_n$ and $x_n \in X$ be such that $\|x_n\| = 1$ and $P_n x_n = x_n$. Let
 $y_n = \sum \limits_{i=0}^{p-1} T^i x_n, n \in \mathds{N}$. Then obviously the elements $y_n$ are pairwise disjoint in $X$, whence $\|y_n\| \geq \|x_n\| = 1$, and a simple estimate (see also~\cite{Ki1}) shows that $Ty_n - y_n \mathop \rightarrow \limits_{n \to \infty} 0$. Thus $1 \in \sigma_2(T)$.

 Next, consider the band $C_n$ in $X$ corresponding to the clopen set $\bigcup \limits_{i=0}^{p-1} \varphi^{(i)}(E_n)$. The band $C_n$ as well as its complementary band are $T$-invariant whence the band $C_n^\prime$ (the Banach dual of $C_n$) in $X^\prime$  is $T^\prime$-invariant. The restriction of the operator $T^p$ on $C_n$ is a central operator on $C_n$ whence $(T^\prime)^p$ is central on $C_n^\prime$. Moreover,
 $\|Q_n^\prime( (T^\prime)^p - I)\| = \|Q_n(T^p - I)\| < 1/n, n \in \mathds{N}$, where $Q_n$ is the band projection on the band $C_n$. Like in the previous paragraph we can construct a sequence $f_n$ of pairwise disjoint elements of $X^\prime$ such that $\|f_n\| = 1$ and $T^\prime f_n - f_n \rightarrow 0$. Therefore $1 \in \sigma_2(T) \cap \sigma_2(T^\prime)= \sigma_1(T)$.

 $(2b)$. There are two non $\varphi$-periodic and non-isolated points $k_1, k_2 \in K$ (it can happen that $k_1 = k_2)$ such that at $k_1$ we have inequalities 2(a) and 2(b) and at $k_2$ - inequalities 4(a) and 4(b) with $\lambda = 1$. Because $k_1$ is a non-isolated and non $\varphi$-periodic point we can find clopen subsets $F_n$ of $K$ such that
 \begin{itemize}
   \item $\varphi^{(i)}(F_n) \cap \varphi^{(j)}(F_n) = \emptyset, -(n+1) \leq i < j \leq n+1$.
   \item If $n \neq m$ then $\big{(} \bigcup \limits_{i=-m-1}^{m+1} \varphi^{(i)}(E_m) \big{)} \cap \big{(} \bigcup \limits_{i=-n-1}^{n+1} \varphi^{(i)}(E_n) \big{)} = \emptyset$.
   \item for any $k \in F_n$ we have $|w_i(k)| \geq 1/2, 0\leq i \leq n+1$ and $|w_i(\varphi^{(-i)}(k)| \leq 2, 1 \leq i \leq n+1$.
 \end{itemize}
Let $x_n$ be an element of the band corresponding to the set $\varphi^{(n)}(F_n)$ such that $\|x_n\| = 1$ and let
$$ y_n = \sum \limits_{i = 0}^{2n} \big{(} 1 - \frac{1}{\sqrt{n}}\big{)}^{|n - i|} T^i x_n .$$
Then the estimates very similar to the ones employed in~\cite{Ki1} show that $\|Ty_n - y_n\| = o(\|y_n\|), n \to \infty$ whence
$1 \in \sigma_2(T)$.

Let us prove now that $1 \in \sigma_2(T^\prime)$. The map $h \rightarrow h^\prime, h \in Z(X)$ defines an isometric isomorphism of $Z(X)$ (which we identify with $C(K)$) onto a subalgebra of $Z(X^\prime)$ (which we identify with $C(K^\prime)$). To this isometric embedding corresponds a continuous surjection $\tau : K^\prime \rightarrow K$, and it is not difficult to see that if $h \in C(K)$, $u,v \in K^\prime$, and $\tau(u) = \tau(v)$ then
$h^\prime(u) = h^\prime(v)$. The map $g \rightarrow (U^{-1})^\prime g U^\prime$ defines an isomorphism of $C(K^\prime)$. Let $\varphi^\prime$ be the homeomorphism of $K^\prime$ onto itself corresponding to this isomorphism. Then it is immediate to see that
$$\tau(\varphi^\prime(u)) = \varphi^{(-1)}(\tau(u)). \eqno{(9)} $$
Let $s \in \tau^{-1}(k_2)$. Let $w_n^\prime(u) = w^\prime(u) w^\prime(\varphi^\prime(u)) \ldots w^\prime((\varphi^\prime)^{(n)}(u), n \in \mathds{N}$. It follows from the fact that at $k_2$ we have inequalities (4a) and (4b) and from $(9)$ that
$$|w^\prime_n(s)| \geq 1 \; n \in \mathds{N} $$
and
$$|w^\prime_n((\varphi^\prime)^{(-n)}(s))| \leq 1 , \; n \in \mathds{N}.$$
It is obvious that $s$ cannot be a $\varphi^\prime$-periodic point in $K^\prime$. Moreover, it is not difficult to see that because $k_2$ is not an isolated point in $K$ the set $\tau^{-1}(k_2)$ must contain points that are not isolated in $K^\prime$. Thus, by what we have already proved
$1 \in \sigma_2(w^\prime U^\prime)$. It remains to notice that $(T)^\prime = U^\prime w^\prime = U^\prime (w^\prime U^\prime) (U^\prime)^{-1}$ whence $1 \in \sigma_2(T^\prime)$.

$(3)$ In this part we prove that $\sigma_1(T) \subseteq \sigma_1(S)$. Let $1 \not \in \sigma_1(S)$. Because by Theorem 22 from~\cite{Ki1}
$\sigma(T) = \sigma(S)$ we can assume without loss of generality that $1 \in \sigma(S)$. We have to consider several possibilities.

\noindent $(3a)$ The operator $I - S$ is semi-Fredholm and $null(I-T) = 0$. In other words we assume that $1 \in \sigma_r(S)$. But by Theorems 20 and 22 from~\cite{Ki1} $\sigma_r(T) = \sigma_r(S)$ whence $1 \not \in \sigma_1(T)$.

\noindent $(3b)$  The operator $I - S$ is semi-Fredholm and $0 < null(I-T) < \infty$. Then conditions II(1) - II(5) of Theorem~\ref{t3} are satisfied. We will keep the notations from the statement of Theorem~\ref{t3}. Then $cl U$ is a clopen $\varphi$-invariant subset of $K$. Let $X_1$ and $X_2$ be the bands in $X$ corresponding to the clopen sets $cl U$ and $K \setminus cl U$, respectively. Clearly $TX_i \subseteq X_i, i=1,2$. Conditions II(1) - II(3) of Theorem~\ref{t3} combined with theorems 20 and 21 from~\cite{Ki1} guarantee that the operator $(I-T)|X_1$ is Fredholm while condition II(5) together with the same theorems implies that $1 \in \sigma_r(T|X_2)$. Thus the operator $(I-T)$ is semi-Fredholm and
$ind(I-T) = ind(I-S)$.

\noindent $(3c)$ The operator $I - S$ is semi-Fredholm and $def(I - S) = 0$. It follows from Theorem~\ref{t2} and from the Frolik's theorem that $K$ can be partitioned as $K = E \cup Q \cup F \cup P$ where $P$ is a clopen $\varphi$-invariant subset of $K$ and there is an $m \in \mathds{N}$ such that $P = \bigcup \limits_{i=1}^m \Pi^p$, while the sets $E$, $F$, and $Q$ satisfy conditions $A$ and $B$ of Theorem~\ref{t1}. Let $K^\prime$, $\tau$, and $\varphi^\prime$ be as in part $(2b)$ of the proof. Let $P^\prime$, $E^\prime$, $F^\prime$, and $Q^\prime$ be the $\tau$-preimages in $K^\prime$ of the corresponding sets in $K$. Then it is easy to see that $P^\prime$ is a clopen $\varphi^\prime$-invariant subset of $K^\prime$ and the sets $E^\prime$, $F^\prime$, and $Q^\prime$ satisfy conditions $A$ and $C$ of Theorem~\ref{t1} (Of course, we have to substitute $K$ by $K^\prime$, $w$ by $w^\prime$, and $\varphi$ by $\varphi^\prime$, respectively). Let $\tilde{S}$ be the operator on $C(K^\prime)$ defined as
$$(\tilde{S})f(s) = w^\prime(s) f(\varphi^\prime(s)), \; f \in C(K^\prime), \; s \in K^\prime .$$
 By Theorem~\ref{t2} $1 \in \sigma_r(\tilde{S})$, and then by part $(3a)$ of the proof we have $1 \in \sigma_r(T^\prime)$ whence
 $(I-T)X = X$.

 \noindent $(3(d)$ The remaining case when $I - S$ is semi-Fredholm and $0 < def(I-T) < \infty$ can be considered similarly to part $(3b)$ of the proof by using part $(3c)$ and Theorem~\ref{t4}.

 Thus we have proved that $\sigma_1(T) = \sigma_1(S)$.

 $(4)$ The arguments applied in parts $(2)$ and $(3)$ of the proof show immediately that $\sigma_i(T) = \sigma_i(S), i=2,3,4$. Finally, the equality
$\sigma_5(T) = \sigma_5(S)$ follows from $\sigma_4(T) = \sigma_4(S)$ and from $\sigma(T) = \sigma(S)$ (~\cite[Theorem 22]{Ki1}).

\end{proof}

\begin{corollary} \label{c4} Let $T$ be an operator of form (6) on a Dedekind complete Banach lattice $X$. Then

$(1)$ If the operator $T$ is band irreducible then $\sigma_1(T) = \sigma(T)$.

$(2)$ If the Banach lattice $X$ has no atoms then $\sigma_3(T) = \sigma(T)$.

\end{corollary}

Our next goal is to extend the result of Theorem~\ref{t8} on arbitrary Banach lattices.

Let $X$ be a Banach lattice, $w \in Z(X)$, and $U$ be a $d$-isomorphism on $X$ such that $\sigma(U) \subseteq \Gamma$. Because the operators $w$ and $U$ are order continuous, by Veksler's theorem~\cite{Ve} (see also~\cite[Theorem 1.65, page 55]{AB}, or~\cite[Lemma 140.1, page 651]{Za} ) they have unique order continuous extensions $\hat{w}$ and $\hat{U}$ to the Dedekind completion $\hat{X}$ of $X$. It is easy to see that $\hat{w} \in Z(\hat{X})$, that $\hat{U}$ is a $d$-isomorphism of $\hat{X}$, and that
$\sigma(\hat{U}) \subseteq \Gamma$. Let $\hat{T} = \hat{w} \hat{U}$. Like in Theorem~\ref{t8} we can consider operator $S$ associated with $\hat{T}$ and defined by $(7)$ where the compact space $\hat{K}$ is the Gelfand compact of $Z(\hat{X})$ and $\hat{\varphi}$ is the homeomorphism of $\hat{K}$ induced by the map $f \to UfU^{-1}, f \in C(\hat{K})$. We will prove below in Theorem~\ref{t9} that
$$\sigma_i(T) = \sigma_i(\hat{T}), i= 1, \ldots, 5 . \eqno{(10)}$$

We start with the following special case of $(10)$

\begin{proposition} \label{p2} Let $K$ be a compact Hausdorff space and $\hat{K}$ be the absolute (or Stonean compact) of $K$. Let $\varphi$ be a homeomorphism of $K$ onto itself, $w \in C(K)$, and $(Tf)(k) = w(k)f(\varphi(k)), k \in K, f \in C(K)$. Let $\hat{T}$ be the unique order continuous extension of $T$ onto $\widehat{C(K)} = C(\hat{K})$. Then $\sigma_i(T) = \sigma_i(\hat{T}), i=1, \ldots, 5$.

\end{proposition}

\begin{proof} $C(K)$ is isometrically and algebraically embedded into $C(\hat{K})$. Let $\tau$ be the surjection of $\hat{K}$ onto $K$ induced by this embedding.  The operator $\hat{T}$ is of the form $(\hat{T}f)(k) = \hat{w}(k)f(\hat{\varphi}(k)), k \in \hat{K}, f \in C(\hat{K})$ and we have that
$$\varphi(\tau(k)) = \tau(\hat{\varphi}(k)), k \in \hat{K}. \eqno{(11)}$$
 The equalities  $\sigma_i(T) = \sigma_i(\hat{T}), i= 1, \ldots 5$ follow easily from $(11)$, Theorem~\ref{t6}, and Proposition~\ref{p1}.

\end{proof}

Our next step is to prove that $(10)$ holds in the case of Banach lattices with a quasi-interior point

Let us recall that a point $u$ in a Banach lattice $X$ is called \textbf{quasi-interior} if the principal ideal $X_u$ is dense in $X$.

\begin{lemma} \label{la1} Let $X$ be a Banach lattice with a quasi-interior point $u$. Let $Z(X) = C(K)$ and $Z(\hat{X}) = C(K_1)$ be the ideal centers of $X$ and its Dedekind completion $\hat{X}$, respectively. Then $C(K_1)$ is isometrically and lattice isomorphic to $\widehat{C(K)} = C(\hat{K})$ where $\hat{K}$ is the absolute of $K$.

\end{lemma}

\begin{proof} The proof below was communicated to the author by A. W. Wickstead.

 First observe that for a Dedekind complete Riesz space with weak order unit $u$ the center is isomorphic to the ideal generated by $u$. If $u$ is a quasi-interior point for a Banach lattice $X$ then it is a weak order unit for the Dedekind completion $\hat{X}$. Observe that the Dedekind completion of the principal ideal $X_u$ may be identified with $\hat{X}_u$. Then we can identify as follows:
$$\widehat{Z(X)}\equiv \widehat{X_u}\equiv \hat{X}_u\equiv Z(\hat{X}).$$
\end{proof}

\begin{theorem} \label{t10} Let $X$ be a Banach lattice with a quasi-interior point. Let $w \in Z(X)$, $U$ be a $d$-isomorphism of $X$ such that
$\sigma(U) \subseteq \Gamma$, and $T = wU$. Let $\hat{T}= \hat{w} \hat{U}$ be the unique order continuous extension of $T$ onto $\hat{X}$. Then
$\sigma_i(T) = \sigma_i(\hat{T}), i= 1, \ldots 5$.

\end{theorem}

\begin{proof}  (1) $0 \in \sigma_i(T) \Leftrightarrow 0 \in \sigma_i(\hat{T}), i= 1, \ldots, 5. $

 \noindent If $0  \in \sigma(\hat{T}) \setminus \sigma_1(\hat{T})$ then there are $n \in \mathds{N}$ and pairwise disjoint atoms $u_1, \ldots , u_n$ in $\hat{X}$ such that $\ker{\hat{T}} = B$ where $B$ is the band in $\hat{X}$ generated by $u_1, \ldots , u_n$. But clearly $B \subseteq X$ and $B$ is a projection band in $X$. Moreover, the operator $T + P_B$, where $P_B$ is the band projection on $B$ is invertible in $X$ whence $T$ is Fredholm and $null(T) = null(\hat{T})$. Then we can prove that $ind(T) = 0$ similar to part (1) of the proof of Theorem~\ref{t8}.

Assume now that $0 \in \sigma_1(\hat{T}) = \sigma_1(S)$. Recall that for any nonzero $\hat{x} \in \hat{X}$ there is a nonzero $x \in X$ such that
$|x| \leq |\hat{x}|$. Then we can see that $0 \in \sigma_1(T)$ in the same way as in the proof of Theorem~\ref{t8}. Moreover, the same reasoning as in the proof of Theorem~\ref{t8} shows that $0 \in \sigma_i(T) \Leftrightarrow 0 \in \sigma_i(\hat{T}), i= 1, \ldots, 5$.

\noindent (2) $ \sigma_1(\hat{T}) \subseteq \sigma_1(T)$. Let $\lambda \in \sigma_1(\hat{T})$. By step (1) we can assume without loss of generality that $\lambda = 1$. By Theorems~\ref{t6} and~\ref{t8} there are non isolated (maybe identical) points $\hat{k}_1, \hat{k}_2 \in \hat{K}$ such that   at these points we have inequalities (2a), (2b) and (4a), (4b), respectively. Let us prove first that $1 \in \sigma_2(T)$. let us first consider the case when $T$ is invertible on $X$. Consider the band in $\hat{X}$ we denoted in the proof of Theorem~\ref{t8} as $C_n$. Here it will be more convenient to denote it as $\hat{C_n}$. Let $C_n = \hat{C_n} \cap X$ be the corresponding band in $X$. Then we can readily see that $TC_n = C_n$. It follows that $T^\prime C_n^\bot =C_n^\bot$ where $C_n^\bot$ is the annihilator of $C_n$ in $X^\prime$. $C_n^\bot$ is a band in $X^\prime$ and its complementary band $D_n$ is isometrically isomorphic to $C_n^\prime$. We see now that $T^\prime D_n = D_n$ and $(T^\prime)^p|D_n \in Z(D_n)$ where $p$ has the same meaning as in the corresponding part of the proof of Theorem~\ref{t8}. Then we can proceed as in the proof of Theorem~\ref{t8}.

Now let us drop the condition that $T$ is invertible. We still can consider the bands $C_n$ in $X$ and $D_n$ in $X^\prime$. Let $R = (|w| + \varepsilon)U$ where $\varepsilon > 0$. The operator $R$ is invertible on $X$ and the previous paragraph shows that $R^\prime D_n = D_n$. The Banach lattice $X^\prime$ is Dedekind complete whence $T^\prime = R^\prime F$ where $F \in Z(X^\prime)$. Therefore $T^\prime D_n \subseteq D_n$ and again we can proceed like in the proof of Theorem~\ref{t8}.

Next we will prove that $1 \in \sigma_2(T^\prime)$. From Lemma~\ref{la1} we easily conclude that there is a non isolated point $ k_2 \in K$ satisfying the inequalities (4a) and (4b), where $K$ is the Gelfand compact of $Z(X)$. After we have noticed it we can repeat the corresponding arguments from the proof of Theorem~\ref{t8}.

\noindent (3) $ \sigma_1(T) \subseteq \sigma_1(\hat{T})$. Let $1 \not \in \sigma_2(\hat{T})$. Then it is obvious from~\cite[Corollary I.4.7, page 43]{EE} that $1 \not \in \sigma_2(T)$. Let $1 \not \in \sigma_2((\hat{T})^\prime)$. Assume contrary to our statement that $1 \in \sigma_2(T^\prime)$.  Then by Theorem~\ref{t6} there is a non isolated point $u \in K^\prime$ (where $K^\prime$ is the Gelfand compact of $Z(X^\prime)$) such that at $u$ we have inequalities (2a) and (2b) for $w^\prime$ and the map $\varphi^\prime$ corresponding to the operator $U^\prime$. From that and the fact that $X$ is a lattice with a quasi interior point follows that there is a non isolated point $k$ in $K$ such that at $k$ we have inequalities (4a) and (4b). Then it follows from Theorems~\ref{t6} and~\ref{t8}, from Proposition~\ref{p2} and from Lemma~\ref{la1} that
$1 \in \sigma_2((\hat{T})^\prime)$, a contradiction.
As far we have proved that $\sigma_i(T) = \sigma_i(\hat{T}), i=1,2$ as well as $\sigma_2(T^\prime) = \sigma_2((\hat{T})^\prime)$

\noindent (4) The equality $\sigma_3(T) = \sigma_3(\hat{T})$ follows from the previous steps and the identity $\sigma_3(T) = \sigma_2(T) \cup \sigma_2(T^\prime)$.

\noindent (5) $\sigma_4(T) = \sigma_4(\hat{T})$. It is enough to prove that if $1 \not \in \sigma_3(T)$ then $ind (I - T) = ind (I - \hat{T})$.
We know from the previous step that the operator $(I - \hat{T})$ is Fredholm. The inequality $null(I-T) \leq null(I-\hat{T})$ is trivial. The inverse inequality $null(I - \hat{T}) \leq null(I-T)$ follows from Theorems~\ref{t6}, ~\ref{t8}, and the trivial fact that the lattices $X$ and $\hat{X}$ have the same supply of atoms. Similarly we obtain the inequality $null(I - (\hat{T})^\prime) \leq null(I-T^\prime)$. Finally the inequality
$null(I - (\hat{T})^\prime) \geq null(I-T^\prime)$ follows from the reasoning applied on step (3).

\noindent (6) $\sigma_5(T) = \sigma_5(\hat{T})$. This statement follows from the definition of the set $\sigma_5(T)$ and the equalities
$\sigma(T) = \sigma(\hat{T})$ and $\sigma_1(T) = \sigma_1(\hat{T})$.

\end{proof}

Now we are ready to prove $(10)$ in full generality.

\begin{theorem} \label{t9} Let $X$ be a Banach lattice. Let $w \in Z(X)$, $U$ be a $d$-isomorphism of $X$ such that
$\sigma(U) \subseteq \Gamma$, and $T = wU$. Let $\hat{T}= \hat{w} \hat{U}$ be the unique order continuous extension of $T$ onto $\hat{X}$. Then
$\sigma_i(T) = \sigma_i(\hat{T}), i= 1, \ldots 5$.

\end{theorem}

\begin{proof} A look at the proof of Theorem~\ref{t10} shows that we used the condition that $X$ has a quasi interior point at two places.

\noindent (I) $1 \in \sigma_2((\hat{T})^\prime) \Rightarrow 1 \in \sigma_2(T^\prime)$. To prove this implication in general case notice first that if $1 \in \sigma_2((\hat{T})^\prime)$ then by Theorems~\ref{t6} and~\ref{t8} there is a non isolated point $u$ in the Gelfand compact $(\hat{K})^\prime$ of $Z((\hat{X})^\prime)$ such that at this point we have inequalities (2a) and (2b) for $(\hat{w})^\prime$ and the homeomorphism $\psi^\prime$ corresponding to $(\hat{U})^\prime$. For any $x \in X$ let $J(x)$ be a closed $\hat{U}$ and$(\hat{U})^{-1}$-invariant ideal in $\hat{X}$ such that $\hat{x} \in J(x)$. Recall that the conjugate $J(x)^\prime$ is a band in $(\hat{X})^\prime$ which we will denote as $B_x$. There are two possibilities.

(a) There is an $x \in X$ such that $1 \in \sigma_2((\hat{T})^\prime | B_x)$. Then we proceed as in the proof of Theorem~\ref{t10}.

(b) If we cannot find an $x$ as in (a) then notice that the set $\bigcup \limits_{x \in X} supp(B_x)$ is dense in $(\hat{K})^\prime$. Therefore we can find elements $x_n \in X$ and points $u_n \in supp(B_{x_n})$ such that the bands $B_{x_n}$ are pairwise disjoint
$$ |(\hat(w_m))^\prime (u_n)| \geq 1- 1/n , m = 1, \ldots , n$$
and
$$ |(\hat(w_m))^\prime ((\psi^\prime)^{(-m)}(u_n))| \leq 1+ 1/n , m = 1, \ldots , n.$$
Applying Lemma~\ref{la1} we can find points $v_n \in K^\prime$ such that
$$ |((w_m)^\prime (u_n)| \geq 1- 1/n , m = 1, \ldots , n$$
and
$$ |((w_m)^\prime ((\varphi^\prime)^{(-m)}(u_n))| \leq 1+ 1/n , m = 1, \ldots , n.$$
Let $v$ be any accumulation point of the set $\{v_n\}$. Then at $v$ we have inequalities (2a) and (2b) whence $1 \in \sigma_2(T^\prime)$.

\noindent (II)$1 \in \sigma_2(T^\prime) \Rightarrow 1 \in \sigma_2((\hat{T})^\prime) $. We can use the arguments from part(3) of the proof of
Theorem~\ref{t10} and the fact that if $I_x$ is a closed $U$-invariant ideal in $X$ and $C_x = (I_x)^\prime$ then $\bigcup \limits_{x \in X} supp(C_x) = K^\prime$.

\end{proof}

\begin{corollary} \label{c5}The statement of Corollary~\ref{c4} remains true without the assumption that $X$ is a Dedekind complete Banach lattice.

\end{corollary}

In connection with Theorem~\ref{t9} the following question might be of interest.

\begin{problem} \label{pr1} Let $X$ be a Banach lattice and $T$ be an order continuous linear bounded operator on $X$. Let $\hat{X}$ be the Dedekind completion of $X$ and $\hat{T}$ be the unique order continuous extension of $T$ on $\hat{X}$.

Is it true that $\sigma_i(T) = \sigma_i(\hat{T}), i = 0, 1, \ldots 5$, where $\sigma_0(T) = \sigma(T)$

\end{problem}

\bigskip

\section{$C(K)$ revisited. The case $(\lambda I - T)C(K) = C(K)$ for weighted compositions generated by non-invertible open surjections.}

The spectrum of arbitrary disjointness preserving operators on $C(K)$ was described in~\cite{Ki}. The problem of describing essential spectra of such operators in the case when the map $\varphi: K \rightarrow K$ is not invertible becomes considerably more complicated and its complete solution remains unknown to the author. In this section we will provide necessary and sufficient conditions for the equality $(\lambda I - wT_\varphi)C(K) = C(K)$ in the case when $|w| > 0$ on $C(K)$, the map $\varphi : K \rightarrow K$ is open, and $\varphi(K) =K$ (see Theorem~\ref{t100} below).

\begin{definition} \label{d300} Let $K$ be a compact Hausdorff space and $\varphi$ be a continuous map of $K$ onto itself.

$(1)$ We call a subset $S$ of $K$ a $\varphi$-string if

 $S = \{s_i : i \in \mathds{Z}\}$ and $\varphi(s_i) = s_{i+1}, i \in \mathds{Z}$

 $(2)$ Let $S$ be a $\varphi$-string. We define the set $\overrightarrow{S}$ as
 $$ \overrightarrow{S} = \bigcap \limits_{i=1}^\infty cl \{s_k : k \geq i\}$$
and the set $\overleftarrow{S}$ as
 $$ \overleftarrow{S} = \bigcap \limits_{i=1}^\infty cl \{s_k : k \leq -i\}.$$

\end{definition}

\begin{lemma} \label{l100} Let $T = wT_\varphi$ be a weighted composition operator on $C(K)$ and $\varphi(K) = K$. Let $\lambda \in \sigma_r(T^\star)$, i.e. $\lambda \in \sigma(T)$ and $(\lambda I - T)C(K) = C(K)$. Then $\lambda \neq 0$.
\end{lemma}

\begin{proof} If $TC(K) = C(K)$ then clearly $|w| > 0$ on $K$ whence the operator of multiplication by $w$ is invertible in $C(K)$ and therefore $T_\varphi C(K) = C(K)$. It follows immediately that $\varphi$ is injective and therefore a homeomorphism of $K$ onto itself whence $T$ is invertible on $C(K)$.

\end{proof}

\begin{lemma} \label{l200} Let $T = wT_\varphi$ be a weighted composition operator on $C(K)$. Let $\varphi(K) = K$ and $|w| >0$ on $K$. Let $\lambda \in \sigma_r(T^\star)$. Then there is an open subset $U$ of $K$ such that

(i) $\varphi(U) = \varphi^{-1}(U) = U$

 and for any $\varphi$ - string $S = \{s_i, i \in \mathds{Z}\}$ such that $S \subset U$ we have
$$ (ii) \; \liminf \limits_{n \to \infty} |w_n(s_0)|^{1/n} > |\lambda| \; \mathrm{and} \; \limsup \limits_{n \to \infty} |w_n(s_{-n})|^{1/n} < |\lambda|. \eqno{(12)} $$
\end{lemma}

\begin{proof} Because $\lambda \in \sigma_r(T^\star)$ there is $f \in C(K)$ such that $f \neq 0$ and $Tf = \lambda f$. Let $U_1 = \{k \in K : f(k) \neq 0\}$.
It follows easily from $|w| > 0$ on $K$ that $\varphi(U_1) = \varphi^{-1}(U_1) = U_1$. Let $K_1 = cl(U_1)$ then $\varphi(K_1) = K_1$ and the formula $T = wT_\varphi$ shows that $T$ induces a weighted composition operator $T_1 = w_1 T_{\varphi_1}$ on $C(K_1)$ where $w_1$ and $\varphi_1$ are restrictions of $w$ and $\varphi$, respectively, on $K_1$. Clearly $f_1 = f|K_1 \in C(K_1)$ and $T_1f_1 = \lambda f_1$. It follows from the fact that $(\lambda I - T)C(K) = C(K)$ and from the Tietze extension theorem that $(\lambda I - T_1)C(K_1) = C(K_1)$ whence $\lambda \in \sigma_r(T_1^\star)$. The set $\sigma_r(T_1^\star)$ is open in $\mathds{C}$ and therefore there is $\gamma \in \sigma_r(T_1^\star)$ such that $|\gamma| > |\lambda|$. Let $g$ be a nonzero function in $C(K_1)$ such that $T_1g = \gamma g$ and let $V = \{k \in K_1 : g(k) \neq 0\}$. Then $V$ is an open subset of $K_1$ and $\varphi_1^{-1}(V) = \varphi_1(V) = V$. We cannot claim that either $V$ is open in $K$ or that $\varphi^{-1}(V) = V$ but it is true that $V \subset \varphi^{-1}(V)$. Consider $U_2 = U_1 \cap V$; then it is immediate that  $U_2$ is an open nonempty subset of $K$ and recalling that $\varphi^{-1}(U_1) = U_1$ and $\varphi_1^{-1}(V) =V$ we see that $\varphi^{-1}(U_2) = \varphi(U_2) = U_2$. Let $K_2 = cl U_2$ and let $T_2$ be the operator on $C(K_2)$ defined similarly to the definition of $T_1$ above. Then
$\lambda \in \sigma_r(T_2^\star)$ and we can find $\delta \in \sigma_r(T_2^\star)$ such that $|\delta| < |\lambda|$. Let $h \in C(K_2)$ be such that $h \neq 0$ and $T_2 h = \delta h$. Let $W = \{k \in K_2 : \; h(k) \neq 0 \}$ and let $U = W \cap U_2$. As above we can show that $U$ is an open subset of $K$ and
$$\varphi^{-1}(U) = \varphi(U) = U \eqno{(13)}$$.

Now let $S \subset U$ be a $\varphi$-string. For any $n \in \mathds{N}$ we have $w_n(s_0)g(\varphi^n(s_0)) = \gamma^n g(s_0)$ whence
 $$|w_n(s_0) \geq |g(s_0)||\gamma|^n/\|g\|_{C(K_1)}. \eqno{(14)} $$
 On the other hand for every $n \in \mathds{N}$ we have $w_n(s_{-n})h(s_0) = \delta^n h(s_{-n})$ whence
 $$ |w_n(s_{-n} \leq |\delta|^n \|h\|_{C(K_2)}/h(s_0) \eqno{(15)}. $$
The statement of the lemma follows from $(13) - (15)$.
\end{proof}

\begin{definition} \label{d500} Assume conditions of Lemma~\ref{l200}. We will denote by $O(\lambda)$ the union of all open subsets of $K$ with properties (i) and (ii) from the statement of Lemma~\ref{l200}. Clearly $O(\lambda)$ is the largest (by inclusion) subset of $K$ with these properties.
\end{definition}

\begin{lemma} \label{l300} Assume conditions of Lemma~\ref{l200}. Let $O(\lambda)$ be from Definition~\ref{d500}. Then the set $K_\lambda = K \setminus O(\lambda)$ is not empty.
\end{lemma}

\begin{proof} If $K_\lambda = \emptyset$ then $O(\lambda)= K$ is a compact Hausdorff space. We can assume without loss of generality that $|\lambda| = 1$. It follows from the definition of $O(\lambda)$ that for every $k \in K$ there are an open neighborhood $V(k)$ of $k$ and $m(k) \in \mathds{N}$  such that $|w_{m(k)}(t)| > 2, t \in V(k)$. Let $\{V(k_1), \ldots , V(k_s)\}$ be a finite subcover of the cover $\{V(k): k \in K\}$ and let $m_i = m(k_i), i=1, \ldots , s$. Next let us fix $k_0 \in K $. Then for any $n \in \mathds{N}$ we can find $p \in \mathds{N}$ such that $0 \leq n - p \leq \max{\{m_1, \ldots, m_s\}}$ and
$|w_p(\varphi^{-n}(k_0)| > 2$ in obvious contradiction with the second inequality in $(12)$.

\end{proof}

\begin{lemma} \label{l400} Assume conditions of Lemma~\ref{l200}. Let $O(\lambda)$ be from Definition~\ref{d500}. Let $T_\lambda$ be the weighted composition operator induced by $T$ on $C(K_\lambda)$ where $K_\lambda = K \setminus O(\lambda)$.

Then $\lambda \not \in \sigma(T_\lambda)$.

\end{lemma}

\begin{proof} First of all notice that because $\varphi^{-1}(O(\lambda)) = O(\lambda)$ we have $\varphi(K_\lambda) = K_\lambda$ and the operator $T_\lambda$
is correctly defined. Moreover, $(\lambda I - T_\lambda) C(K_\lambda) = C(K_\lambda)$. Assume contrary to the statement of the lemma that $\lambda \in \sigma(T_\lambda)$. Then $\lambda \in \sigma_r(T_\lambda ^\star)$ and by Lemma~\ref{l200} there is an open subset $V$ of $K_\lambda$ with properties (i) and (ii) from the statement of that lemma. The set $O(\lambda) \cup V$ is open in $K$ and has properties (i) and (ii) in contradiction with maximality of $O(\lambda)$.
\end{proof}

\begin{lemma} \label{l500} Assume conditions of Lemma~\ref{l200}. Let $F_\lambda = cl O(\lambda) \setminus O(\lambda)$. Let $T^\lambda$ be the operator induced by $T$ on $C(F_\lambda)$.

Then $\sigma(T^\lambda, C(F_\lambda)) \cap \lambda \Gamma = \emptyset$.

\end{lemma}

\begin{proof} We will prove first that $\lambda \not \in \sigma(T^\lambda)$. Indeed, because $\varphi(F_\lambda) = F_\lambda$ we have
$(\lambda I - T^\lambda)C(F_\lambda) = C(F_\lambda)$ and therefore, if $\lambda \in \sigma(T^\lambda)$, then there is a nonzero $g \in C(F_\lambda)$ such that
$T^\lambda g = \lambda g$. But $F_\lambda$ is a closed subset of $K_\lambda$ and therefore $g$ can be identified with an element $g^{ \star \star}$ in
$C(K_\lambda)^{\star \star}$ such that $T_\lambda^{ \star \star} g^{ \star \star} = \lambda g^{\star \star}$ in contradiction with Lemma~\ref{l400}.

Next assume that there is $\gamma \in \sigma(F_\lambda)$ such that $|\gamma| = |\lambda|$. Then by Theorem 3.12 in~\cite{Ki} there is a $\varphi$-periodic point $t \in F_\lambda$ such that $\gamma^p = w_p(t)$ where $p$ is the smallest positive period of $t$. It follows from $(1)$ and from $\varphi(O(\lambda)) = O(\lambda)$ that the set $O(\lambda)$ cannot contain eventually \footnote{A point $k \in K$ is called eventually $\varphi$-periodic if there is an $n \in \mathds{N}$ such that the point $\varphi^{(n)}(k)$ is $\varphi$-periodic.}  $\varphi$-periodic points. Then it follows from the proof of Lemma 3.5 in~\cite{Ki} that
$\lambda \Gamma \subset \sigma_{ap}(T^\star)$, a contradiction.
\end{proof}

\begin{lemma} \label{l600} There is a closed subset $G_\lambda$ of $F_\lambda$ with the following properties.
\begin{enumerate}
  \item $\varphi(G_\lambda) = G_\lambda$.
  \item The restriction of $\varphi$ on $G_\lambda$ is a homeomorphism of $G_\lambda$ onto itself.
  \item The operator $T_{G_\lambda}$ defined on $C(G_\lambda)$ by the formula $T(f|G_\lambda) = (Tf)|G_\lambda$ is invertible and
  $\rho(T_{G_\lambda}^{-1}) < 1/|\lambda|$.
  \item $\exists m \in \mathds{N}$ such that $G_\lambda \subset Int_{F_\lambda} \varphi^{-m}(G_\lambda)$.
  \item Let $H_\lambda = F_\lambda \setminus \bigcup \limits_{n = 1}^\infty \varphi^{-n}(G_\lambda)$. Then $H_\lambda$ is a closed subset of $F_\lambda$,
  $\varphi(H_\lambda) = H_\lambda$, and $\rho(T_{H_\lambda}) < |\lambda|$.
\end{enumerate}

\end{lemma}

\begin{proof} The proof follows immediately from Lemma~\ref{l500} and Theorem 3.10 in~\cite{Ki}.

\end{proof}

\begin{lemma} \label{l700} Let the map $\varphi$ be open and let $G_\lambda$ be the set from the statement of Lemma~\ref{l500}. Then there is an open neighborhood $V$ of $G_\lambda$ in $K$ such that the map $\varphi : V \to K$ is one-to-one.

\end{lemma}

\begin{proof} We can assume without loss of generality that $\lambda = 1$. For any $n \in \mathds{N}$ let $Q_n = \{k \in K : \; |w_{n+1}(k)| > 2\}$. It follows from the inequality $\rho(T_{G_\lambda}^{-1}) < 1$ that for any large enough $n \in \mathds{N}$ the set $Q_n$ is an open neighborhood of $G_\lambda$. Let $R_n = \varphi^n(Q_n)$, then $R_n$ is an open neighborhood of $G_\lambda$ because $\varphi$ is open. We claim that for any large enough $n \in \mathds{N}$ the map
$\varphi : R_n \to K$ is one-to-one. Assume to the contrary that for any $N \in \mathds{N}$ there are an $n > N$ and $p, q \in R_n$ such that $p \neq q$ but $\varphi(p) = \varphi(q)$. Let $s, t \in Q_n$ be such that $\varphi^n(s) = p$ and $\varphi^n(t) = q$. Let $\mu_1 = \frac{1}{w_{n+1}(s)}\delta_s$ and
 $\mu_2 = \frac{1}{w_{n+1}(t)}\delta_t$. Notice that $\|\mu_i\| \leq 1/2, i=1,2$, because $s,t \in Q_n$. Consider the discrete measure $\mu$ on $K$ defined as
$$ \mu = \sum \limits_{i=0}^n \big{(} 1 - \frac{1}{\sqrt{n}}\big{)}^{n-i} (T^\star)^i \mu_1 - \sum \limits_{i=0}^n \big{(} 1 - \frac{1}{\sqrt{n}}\big{)}^{n-i} (T^\star)^i \mu_2. $$
Notice that $(T^\star)^n \mu_1 = \frac{1}{w(p)}\delta_p$ and $(T^\star)^n \mu_2 = \frac{1}{w(q)}\delta_q$ whence $\|\mu\| \geq 2 \|1/w\|_\infty$. Next notice that $(T^\star)^{n+1} \mu_1 = (T^\star)^{n+1} \mu_2 = \delta_{\varphi(p)}$ whence
$$T^\star \mu - \mu = \big{(} 1 - \frac{1}{\sqrt{n}}\big{)}^n (\mu_1 - \mu_2) -\frac{1}{\sqrt{n}} \sum \limits_{i=1}^n  \big{(} 1 - \frac{1}{\sqrt{n}}\big{)}^{n-i} (T^\star)^i \mu_1$$
$$ +\frac{1}{\sqrt{n}} \sum \limits_{i=1}^n \big{(} 1 - \frac{1}{\sqrt{n}}\big{)}^{n-i} (T^\star)^i \mu_2. $$
Therefore $\|T^\star \mu - \mu\| \leq \big{(} 1 - \frac{1}{\sqrt{n}}\big{)}^n + \frac{1}{\sqrt{n}}\|\mu\|$. Because $n$ can be chosen arbitrary large we see that $1 \in \sigma_{a.p.}(T^\star)$, a contradiction.
\end{proof}

\begin{lemma} \label{l800} Assume conditions of Lemma~\ref{l600}. For any $n \in \mathds{N}$ let $Q_n$ be the set defined in the proof of Lemma~\ref{l700}.
Then there is $N \in \mathds{N}$ such that if $n \geq N$ then $\varphi^n(Q_n) \cap O(\lambda) \subset Q_n \cap O(\lambda)$.
\end{lemma}

\begin{proof} Assume contrary to our statement that for any $N \in \mathds{N}$ there are $n > N$ and $k \in O(\lambda)$ such that
 $$|w_n(k)| > 2 \; \mathrm{and} \; |w_n(\varphi^n(k))| \leq 2. \eqno{(16)}$$
 Let $\nu = \frac{1}{w_n(k)}\delta_k$ and
$$ \mu = \sum \limits_{i=0}^{2n-1} \big{(} 1- \frac{1}{\sqrt{n}}\big{)}^{|n-i|}(T^\star)^i \nu.$$
Notice that because $k$ is not a $\varphi$-periodic or eventually $\varphi$-periodic point all the terms in the sum above represent pairwise disjoint point measures on $K$. Therefore $\|\mu\| \geq \|(T^\star)^n \nu\| = \|\delta_{\varphi^n(k)}\| = 1$. On the other hand
$$T^\star \mu - \mu = \big{(} 1- \frac{1}{\sqrt{n}}\big{)}^n \nu - \sum \limits_{i=1}^n \frac{1}{\sqrt{n}}\big{(} 1- \frac{1}{\sqrt{n}}\big{)}^{n-i}(T^\star)^i \nu$$
$$ + \sum \limits_{i=n+1}^{2n-1}\frac{1}{\sqrt{n}}\big{(} 1- \frac{1}{\sqrt{n}}\big{)}^{i-n}(T^\star)^i \nu + \big{(} 1- \frac{1}{\sqrt{n}}\big{)}^{n-1}(T^\star)^{2n}\nu $$
and in virtue of $(16)$ $\|T^\star \mu - \mu\| \leq \big{(} 1- \frac{1}{\sqrt{n}}\big{)}^{n} + 2\big{(} 1- \frac{1}{\sqrt{n}}\big{)}^{n-1}
+ \frac{1}{\sqrt{n}}\|\mu\|$. Because $n$ is arbitrary large we have $1 \in \sigma_{ap}(T^\star)$ in contradiction with our assumption.
\end{proof}

\begin{corollary} \label{c100} There are an open nonempty neighborhood $V$ of $G_\lambda$ in $cl(O_\lambda)$ and $n \in \mathds{N}$ such that
\begin{itemize}
  \item $\varphi(V) \subset V$.
  \item The map $\varphi : V \to \varphi(V)$ is a homeomorphism.
  \item $cl V \cap H_\lambda = \emptyset$.
  \item $|w_n| > 1$ on $cl V$.
\end{itemize}

\end{corollary}

\begin{proof} Fix a large enough $n \in \mathds{N}$ and consider the set $R_n$ introduced in the proof of Lemma~\ref{l700}. The set
$V = \bigcap \limits_{k=0}^{n-1} \varphi^k(R_n \cap O(\lambda))$ has the required properties.

\end{proof}

\begin{lemma} \label{l900} Assume conditions of Lemma~\ref{l600}. Let $V$ be an open neighborhood in $K$ of the set $G_\lambda$ such that $\varphi(V) \subset V$ and $V \cap H_\lambda = \emptyset$. Then $H = cl O_\lambda \setminus \bigcup \limits_{k=1}^\infty \varphi^{-k}(V) = H_\lambda$.
\end{lemma}

\begin{proof} Clearly $\varphi(H) \subseteq H$ and $H_\lambda \subseteq H$. Assume that $H \setminus H_\lambda \neq \emptyset$. It follows from $(5)$ in the statement of Lemma~\ref{l600} that $H \setminus H_\lambda \subset O(\lambda)$. Let $T_H$ be the weighted composition operator $wT_\varphi$ considered on $C(H)$. We have to consider three possibilities.

$(a)$ $\rho(T_H) < |\lambda|$. That contradicts the first inequality in $(1)$ in the statement of Lemma~\ref{l200}.

$(b)$ $\rho(T_H) > |\lambda|$ and $\lambda \not \in \sigma(T_H)$. In this case it follows from~\cite{Ki} that there is a closed subset $L$ of $H$ such that
$\varphi(L) = L$, the operator $T_L = wT_\varphi$ is invertible on $C(L)$, and $\rho(T_L^{-1} < 1/|\lambda|$. Clearly $L \cup H_\lambda = \emptyset$ whence
$L \subset O(\lambda)$. But then we come to a contradiction with the second inequality in $(1)$.

$(c)$ $\lambda \in \sigma_r(T^\star_H)$. We bring this case to a contradiction similarly to $(b)$ by using statements $(1) - (3)$ of Lemma~\ref{l600}.
\end{proof}

The previous statements provide necessary conditions for $\lambda \in \sigma_r(T^\star)$. As Theorem~\ref{t100} below shows the combination of these conditions is also sufficient. Before we state and prove this theorem we need one simple result which is most probably known.

\begin{lemma} \label{l1000} Let $K$ be a compact Hausdorff space, $\varphi$ be a map of $K$ into itself, and $T_\varphi$ be the corresponding composition operator on $C(K)$. The following statements are equivalent
\begin{enumerate}
  \item The Banach dual operator $T_\varphi^\star$ preserves disjointness.
  \item The map $\varphi$ is one-to-one.
\end{enumerate}

\end{lemma}

\begin{proof} The implication $(1) \Rightarrow (2)$ is trivial. Assume that $\varphi$ is a homeomorphism of $K$ onto $\varphi(K)$. The operator $T_\varphi$ induces a positive isometry of $C(\varphi(K))$ onto $C(K)$. Therefore the dual operator $T_\varphi^\star$ can be considered as a positive isometry of $C(K)^\star$ onto
 $C(\varphi(K))^\star$. To finish the proof we can use a theorem of Abramovich~\cite{Ab} that states that a positive surjective isometry between normed lattices preserves disjointness.  Alternatively we can notice that $C(K)^\star$ and $C(\varphi(K))^\star$ are $L^1$-spaces and we can apply the theorem of Lamperti (see e.g.~\cite[Chapter 3]{FJ}) to see that $T^\star_\varphi$ preserves disjointness (because in the statement of Lamperti theorem the measure is assumed to be sigma-finite some simple additional reasoning is needed).
\end{proof}

\begin{theorem} \label{t100} Let $K$ be a compact Hausdorff space and $\varphi$ be an open continuous map of $K$ onto itself. Let $w$ be an invertible element of $C(K)$. Let $T$ be the weighted composition operator
$$ (Tf)(k) = w(k)f(\varphi(k), \; f \in C(K), \; k \in K. $$
Let $\lambda \in \sigma(T)$.
The following conditions are equivalent.
$(I)$ $\lambda \in \sigma_r(T^\star)$ (i.e. $\lambda \in \sigma(T)$ and $(\lambda I - T)C(K) = C(K)$).
$(II)$
\begin{enumerate}
  \item $\lambda \neq 0$
  \item There is a nonempty open subset $O$ of $K$ such that $\varphi(O) = O = \varphi^{-1}(O)$, for every point $k \in O$ conditions $(1)$ are satisfied,
   $F = cl O \setminus O \neq \emptyset$, and $\lambda \not \in \sigma(T,C(K \setminus O)$.
  \item $\lambda \Gamma \cap \sigma(T,C(F)) = \emptyset$.
  \item There are subsets $G$ and $H$ of $F$ with properties $(1) - (5)$ from the statement of Lemma~\ref{l600}.
  \item There are an open neighborhood $V$ of $G$ in $cl O$ and $m \in \mathds{N}$ such that $V \cap H = \emptyset$, $\varphi(V) \subset V$, $cl O \setminus \bigcup \limits_{n=1}^\infty \varphi^{-n}(V) = H$, the map $\varphi : V \to \varphi(V)$ is a homeomorphism, and $|w_m| > 1$ on $cl V$.
\end{enumerate}

\end{theorem}

\begin{proof} The implication $(I) \Rightarrow (II)$ has been already proved.

Assume $(II)$. It follows from $II(1)$ and Theorem 3.10 in~\cite{Ki} that $\lambda \in \sigma(T)$. Assume contrary to our statement that
$\lambda \in \sigma_{ap}(T^\star)$. We can assume without loss of generality that $\lambda =1$. Then there is a sequence $\mu_n \in C(K)^\star, n \in \mathds{N}$ such that $\|\mu_n\|=1$ and $T^\star \mu_n - \mu_n \mathop \rightarrow \limits_{n \to \infty} 0$.

It follows from $II(2)$ that $|\mu_n | (K \setminus O) \mathop \rightarrow \limits_{n \to \infty} 0$ and therefore we can assume that
$|\mu_n| (K \setminus O) =0, n \in \mathds{N}$.

Consider the set $V$ and the integer $m$ from $II(5)$. The ideal $J$ of all functions from $C(K)$ that are equal $0$ on $V$ is $T$-invariant and it is easy to see from $II(5)$ and the fact that $\rho(T,C(H)) <1$ that $\rho(T|J) < 1$. Therefore $|\mu_n|(cl O \setminus V) \mathop \rightarrow \limits_{n \to \infty} 0$ and we can assume that $|\mu_n|(V) =1, n \in \mathds{N}$. Let $T_V$ be the operator $wT_\varphi$ considered on $C(cl V)$. Then $T_V^\star \mu_n - \mu_n \mathop \rightarrow \limits_{n \to \infty} 0$. By Lemma~\ref{l1000} the operator $T_V^\star$ preserves disjointness and therefore $|T_V|^\star |\mu_n| - |\mu_n | \mathop \rightarrow \limits_{n\to \infty} 0$. Let $\mu$ be a probability measure on $V$ which is an accumulation point of the sequence $|\mu_n|, n \in \mathds{N}$ in the weak-$\star$ topology. Then $|T_V^\star| \mu = \mu$. Let $S = supp \, \mu$. Then $\varphi(S) = S$ and the operator $T_S$ induced by $T$ on $C(S)$ is invertible. We have $(T_\varphi^m)^\star (|w_m|)^\star \mu = \mu$. But $(T_\varphi^m)^\star$ is an isometry on $C(S)^\star$ and $\|(|w_m|)^\star \mu \| > \|\mu\|$, a contradiction.
\end{proof}

\section{Appendix}
The purpose of this appendix is to clarify some details about the statement and the proof of Theorem 22 in~\cite{Ki1} which was extensively used in the current paper.

(1) The aforementioned theorem states that if $T$ is an operator on a Banach lattice $X$ of the form (6)then
$$ \sigma(T,X) = \sigma(\hat{T}, \hat{X})= \sigma(S,C(K)) = \sigma(\hat{S}, C(\hat{K})).$$
While the equalities
$$ \sigma(T,X) = \sigma(\hat{T}, \hat{X})=  \sigma(\hat{S}, C(\hat{K}))$$
and their proof in~\cite{Ki1} are correct, the statement that any of these sets is equal to $ \sigma(S,C(K))$ is in general false, as in particular follows from Example 32 in~\cite{Ki1}.

Fortunately, we did not use this equality in full generality, and in the case when $X$ is a Banach lattice with a quasi-interior point it is true as follows from Proposition~\ref{p2} and Lemma~\ref{la1}.

(2) In the proofs of Theorems 22 and 15 in~\cite{Ki1} the following fact was used. Let $X$ be a Banach lattice and $B$ be a band in $X$, then
$$\widehat{X/B} = \hat{X}/\hat{B.}$$
It was (and still is) my assumption that this fact must be known but I was not able to find it in the literature. Therefore a short proof is provided below.

\begin{proof} Consider the canonical map $T : X \to X/B$; $Tx =[x]$. Then~\cite[Problem 9.3.2, page 29]{AA} the map $T$ is order continuous. Therefore by Veksler's theorem~\cite{Ve} the operator $T$ has the unique order continuous extension $\hat{T} : \hat{X} \to \widehat{X/B}$. Notice that because $\hat{T}$ is order continuous and $T(X) = X/B$ we have $\hat{T}(\hat{X})= \widehat{X/B}$. On the other hand
$\ker{\hat{T}} = \hat{B}$. Indeed, obviously $B \subset \ker{\hat{T}}$ and because $\hat{T}$ is order continuous we have $\hat{B} \subset \ker{\hat{T}}$. On the other hand if $\hat{x} \in \hat{X} \setminus \hat{B}$ then there is $x \in B^d$ such that $x \neq 0$ and $|x| \leq |\hat{x}|$. The operators $T$ and (therefore) $\hat{T}$ preserve disjointness whence $|\hat{T}\hat{x}| =\hat{T}|\hat{x}| \geq T|x| =|Tx| \neq 0$ and $\hat{x} \not \in \ker{\hat{T}}$.

Thus we see that $\widehat{X/B}$ is isometrically  and lattice isomorphic to $\hat{B}^d = \hat{X}/\hat{B}$.

\end{proof}

\end{document}